\documentclass[review,reqno]{siamart0216}

\usepackage{lipsum}
\usepackage{amsfonts}
\usepackage{graphicx}
\usepackage{subfigure}
\usepackage{epsfig}
\usepackage{epstopdf}
\usepackage{algorithmic}
\usepackage{multirow}
\usepackage{color}
\usepackage{algorithm}  
\usepackage{algorithmic}
\allowdisplaybreaks
\nolinenumbers

\numberwithin{equation}{section}   

\ifpdf
  \DeclareGraphicsExtensions{.eps,.pdf,.png,.jpg}
\else
  \DeclareGraphicsExtensions{.eps}
\fi

\newcommand{\TheTitle}{Energy Stability and Convergence of SAV
 Block-centered Finite Difference Method for Gradient
Flows}
\newcommand{\TheAuthors}{~~}

\headers{Energy Stability and Convergence}{\TheAuthors}

\title{{\TheTitle}\thanks{Received by editors March 2, 2018 }}

    \author{ Xiaoli Li
        \thanks{School of Mathematics, Shandong University, Jinan 250100, China. Email: xiaolisdu@163.com.}
        \and Jie Shen  \thanks{Corresponding Author. Department of Mathematics, Purdue University, West Lafayette, IN 47907, USA. Email: shen7@purdue.edu.}
       \and Hongxing Rui
        \thanks{School of Mathematics, Shandong University, Jinan 250100, China. Email: hxrui@sdu.edu.cn.}    
        }

\usepackage{amsopn}

\ifpdf
\hypersetup{
  pdftitle={\TheTitle},
  pdfauthor={\TheAuthors}
}
\fi

\usepackage{cases}
\begin{document}

\maketitle

\begin{abstract}
We present in this paper  construction and analysis of a  block-centered finite difference   method  for the  spatial discretization of the scalar auxiliary variable Crank-Nicolson  scheme   (SAV/CN-BCFD) for gradient flows, and show rigorously that scheme is second-order in both time and space in various discrete norms. When 
equipped with an adaptive time strategy, the SAV/CN-BCFD scheme is  accurate and extremely efficient. Numerical experiments on typical Allen-Cahn and Cahn-Hilliard equations are presented to verify our theoretical results and to show the robustness and accuracy of the SAV/CN-BCFD scheme.
\end{abstract}



\pagestyle{myheadings}
\thispagestyle{plain}

\section{Introduction}
Gradient flows are widely used in mathematical models for problems in many fields of science and engineering, particularly in materials science and fluid dynamics, cf. \cite{cahn1958free,cahn1959free,yue2004diffuse,shen2018scalar} and the references therein.  Therefore it is important to develop efficient and accurate numerical schemes for their simulation. There exists an extensive literature on the numerical analysis of gradient flows, see for instance \cite{Chen2015Convergence,grun2013convergent,Diegel2013Analysis,Feng2006Fully,shen2010phase,Elliott1989A,Guo2016An} and the references therein. 

In the algorithm design of gradient flows, an important goal is to guarantee the energy stability at the discrete level, in order to  capture the correct long-time dynamics of the system and provide enough flexibility for dealing with the stiffness problem induced by the thin interface. Many schemes for gradient flows are based on the traditional fully-implicit or explicit discretization for the nonlinear term, which may suffer from  harsh time step constraint due to  the thin interfacial width \cite{feng2003numerical,shen2010numerical}. In order to deal with this problem, the convex splitting approach \cite{shen2012second,wang2011energy, hu2009stable} and linear stabilization approach \cite{liu2007dynamics,shen2010numerical,xu2006stability,zhao2016energy} have been widely used to construct unconditionally energy stable schemes. However, the convex splitting approach usually leads to nonlinear schemes and linear stabilization approach is usually limited to first-order  accuracy.

Recently, a novel numerical method, the so called  invariant energy quadratization (IEQ),  was
proposed  in \cite{MR3564340,zhao2017novel,yang2017numerical}. This method is a generalization of the method of Lagrange multipliers or of auxiliary variable. The IEQ approach is remarkable as it permits us to construct linear, unconditionally stable, and second-order unconditionally energy stable schemes for a large class of gradient flows.
However, it leads to coupled systems with variable coefficients  that may be difficult or expensive to solve. The scalar auxiliary variable (SAV) approach \cite{shen2018scalar,shen2017new} was inspired by the IEQ approach, which inherits its main advantages  but overcomes many of its shortcomings. In particular, in a recent paper \cite{shen2018convergence}, the authors established the first-order convergence and error estimates for the semi-discrete  SAV scheme.

In this paper, we construct a  SAV/CN scheme with block-centered finite differences  for gradient flows, carried out a rigorous stability and error analysis, and implemented an adaptive time stepping strategy  so that the time step is only dictated by accuracy rather than by stability.  The block-centered finite difference method can be thought as the lowest order Raviart-Thomas mixed element method with a suitable quadrature. Its main advantage over using a regular finite difference method  is  that it can approximate both the phase function and chemical potential with   Neumann boundary conditions in the mixed formulation to second-order accuracy, and it guarantees local mass conservation. 
Our approach for error estimates here is very different from that in \cite{shen2018convergence} which is based on deriving $H^2$ bounds for the numerical solution. However, this approach can not be used in the fully discrete case with finite-differences in space. 
The essential tools used in the proof are the summation-by-parts formulae both in space and time to derive energy stability, and  an induction process to show that the discrete $L^\infty$ norm of the numerical solution is uniformly bounded, without assuming a uniform Lipschitz condition on the nonlinear potential. To the best of the authors' knowledge, this is the first paper with rigorous proof of second-order convergence both in time and space for a linear scheme to a class of  gradient flows without assuming a uniform Lipschitz condition for the nonlinear potential. 
 
The paper is organized as follows. In Section 2,  we  describe our numerical scheme, including the temporal discretization and spacial discretization.
In Section 3,  we demonstrate the  energy stability for our SAV/CN-BCFD scheme.
In Section 4, we carry out error estimates for the  SAV/CN-BCFD schemes.
In Section 5, we present some numerical experiments to verify the energy stability and accuracy of the proposed  schemes.

Throughout the paper we use $C$, with or without subscript, to denote a positive
constant, which could have different values at different places.

\section{The SAV/CN-BCFD scheme}
Given a typical energy functional \cite{shen2018convergence}:
\begin{equation}\label{definition of energy}
\aligned
E(\phi)=\int_{\Omega}(\frac{\lambda}{2}\phi^2+\frac{1}{2}|\nabla \phi|^2)d\textbf{x}+E_1(\phi),
\endaligned
\end{equation} 
where $\Omega$ is a rectangular domain in $\mathbb{R}^2$, $\lambda\geq 0$ and $E_1(\phi)=\int_{\Omega}F(\phi)d\textbf{x} \ge -c_0$ for some $c_0>0$, i.e., it is bounded from below.
We consider the following gradient flow:
 \begin{equation}\label{e_model}
  \left\{
   \begin{array}{l}
   \displaystyle\frac{\partial \phi}{\partial t}=M\mathcal{G}\mu,~~~~~~~~~~~~~in~\Omega\times J,\\
   \displaystyle \mu=-\Delta \phi+\lambda \phi+F^{\prime}(\phi),~~~~in~\Omega\times J,
   \end{array}
   \right.
  \end{equation}
 $J=(0,T]$, and $T$ denotes the final time. $M$ is the mobility constant which is positive. The chemical potential $\mu=\frac{\delta E}{\delta \phi}$. $\mathcal{G}=-1$ for the $L^2$ gradient flow and $\mathcal{G}=\Delta$ for the $H^{-1}$ gradient flow. $F(\phi)$ is the nonlinear free energy density and we focus on  as an example, when $E_1(\phi)=\int_{\Omega}\alpha(1-\phi^2)^2d\textbf{x}$, the $L^2$ and $H^{-1}$ gradient flows are the well-known Allen-Cahn and Cahn-Hilliard equations, respectively. 

 The boundary and initial conditions are as follows.
\begin{equation}\label{e5}
  \left\{
   \begin{array}{l}
   \partial_\textbf{n}\phi|_{\partial \Omega}=0,~~\partial_\textbf{n}\mu|_{\partial \Omega}=0,\\
   \phi|_{t=0}=\phi_0,
   \end{array}
   \right.
  \end{equation}
  where $\textbf{n}$ is the unit outward normal vector of the domain $\Omega$. The equation satisfies the following energy dissipation law:
 \begin{equation}\label{energy dissipation of energy}
\aligned
\frac{dE}{dt}=\int_{\Omega}\frac{\partial \phi}{\partial t}\mu d\textbf{x}=
M\int_{\Omega}\mu\mathcal{G}\mu d\textbf{x}\leq 0.
\endaligned
\end{equation}  
  \subsection{The semi discrete SAV/CN scheme}
 
 We recall the SAV/CN scheme introduced in \cite{shen2018scalar} first.
   
 Let $C_0>c_0$ so that $E_1(\phi)+C_0>0$. Without loss of generality, we substitute
 $E_1$ with $E_1+C_0$ without changing the gradient flow. Then $E_1$ has a positive 
 lower bound $\hat{C}_0=C_0-c_0$, which we still denote as $C_0$ for simplicity.

  In the SAV approach, a scalar variable $r(t)=\sqrt{E_1(\phi)}$ is introduced, and  the  
 system (\ref{e_model})  can be transformed into:
  \begin{numcases}{}
  \frac{\partial \phi}{\partial t}=M\mathcal{G}\mu,\label{e_true solution1}\\
   \mu=-\Delta \phi+\lambda \phi+\frac{r}{\sqrt{E_1(\phi)}}F^{\prime}(\phi),\label{e_true solution2}\\
   r_t=\frac{1}{2\sqrt{E_1(\phi)}}\int_{\Omega}F^{\prime}(\phi)\phi_t d\textbf{x},
  \label{e_true solution3} 
\end{numcases}
 Then, the SAV/CN scheme is given as follows:
 \begin{numcases}{}
 \frac{\phi^{n+1}-\phi^n}{\Delta t}=M\mathcal{G}\mu^{n+1/2}, \label{e_semi-discret1}\\
 \mu^{n+1/2}=-\Delta \phi^{n+1/2}+\lambda \phi^{n+1/2}
+\frac{r^{n+1/2}}{\sqrt{E_1(\tilde{\phi}^{n+1/2})}}F^{\prime}(\tilde{\phi}^{n+1/2}),\label{e_semi-discret2}\\
 \frac{r^{n+1}-r^n}{\Delta t}=\frac{1}{2\sqrt{E_1(\tilde{\phi}^{n+1/2})}}\int_{\Omega}F^{\prime}(\tilde{\phi}^{n+1/2}) \frac{\phi^{n+1}-\phi^n}{\Delta t} d\textbf{x},\label{e_semi-discret3}
\end{numcases}
where $\phi^{n+1/2}=\frac{1}{2}(\phi^n+\phi^{n+1})$, $r^{n+1/2}=\frac{1}{2}(r^n+r^{n+1})$,
$\tilde{\phi}^{n+1/2}$ can be any explicit approximation of $\phi(t^{n+1/2})$ with 
an error of $O(\Delta t^2)$. For instance, we may let $\tilde{\phi}^{n+1/2}$ be the extrapolation by
 \begin{equation}\label{e_explicit approx}
\aligned
\tilde{\phi}^{n+1/2}=\frac{1}{2}(3\phi^n-\phi^{n-1}).
\endaligned
\end{equation}  

  \subsection{Spacial discretization} 
 we apply the BCFD method on the staggered grids for the spacial discretization.
  
First we give some preliminaries.
Let $L^m(\Omega)$ be the standard Banach space with norm
$$\| v\|_{L^m(\Omega)}=\left(\int_{\Omega}| v|^md\Omega\right)^{1/m}.$$
For simplicity, let
$$(f,g)=(f,g)_{L^2(\Omega)}=\int_{\Omega}fg~d\Omega$$
denote the $L^2(\Omega)$ inner product,
 $\|v\|_{\infty}=\|v\|_{L^{\infty}(\Omega)}.$ And $W^{k,p}(\Omega)$ be the standard Sobolev space
$$W^{k,p}(\Omega)=\{g:~\| g\|_{W_p^k(\Omega)}<\infty\},$$
where
\begin{equation}\label{enorm1}
\| g\|_{W^{k,p}(\Omega)}=\left(\sum\limits_{|\alpha|\leq k}\| D^\alpha g\|_{L^p(\Omega)}^p \right)^{1/p}.
\end{equation}

The grid points are denoted by
$$(x_{i+1/2},y_{j+1/2}),~~~i=0,...,N_x,~~j=0,...,N_y,$$
and the notations similar to those in \cite{weiser1988convergence} are used.
\begin{equation*}
\aligned
&x_{i}=(x_{i-\frac{1}{2}}+x_{i+\frac{1}{2}})/2,~~~i=1,...,N_x,\\
&h_x=x_{i+\frac{1}{2}}-x_{i-\frac{1}{2}},~~~i=1,...,N_x, \\
&y_{j}=(y_{j-\frac{1}{2}}+y_{j+\frac{1}{2}})/2,~~~j=1,...,N_y,\\
&h_y=y_{j+\frac{1}{2}}-y_{j-\frac{1}{2}},~~~j=1,...,N_y,
\endaligned
\end{equation*}
where $h_x$ and $h_y$ are grid spacings in $x$ and $y$ directions, and $N_x$ and 
$N_y$ are the number of grids along the $x$ and $y$ coordinates, respectively.

Let $g_{i,j},~g_{i+\frac{1}{2},j},~g_{i,j+\frac{1}{2}}$ denote $g(x_{i},y_{j}),~g(x_{i+\frac{1}{2}},y_{j}),~g(x_{i},y_{j+\frac{1}{2}}). $ Define the discrete inner products and norms as follows,

\begin{equation*}
\aligned
&(f,g)_{m}=\sum\limits_{i=1}^{N_{x}}\sum\limits_{j=1}^{N_{y}}h_xh_yf_{i,j}g_{i,j},\\
&(f,g)_{x}=\sum\limits_{i=1}^{N_{x}-1}\sum\limits_{j=1}^{N_{y}}h_xh_yf_{i+\frac{1}{2},j}g_{i+\frac{1}{2},j},\\
&(f,g)_{y}=\sum\limits_{i=1}^{N_{x}}\sum\limits_{j=1}^{N_{y}-1}h_xh_yf_{i,j+\frac{1}{2}}g_{i,j+\frac{1}{2}},\\
&(\textit{\textbf{v}},\textit{\textbf{r}})_{TM}=(v_1,r_1)_{x}+(v_2,r_2)_{y}.
\endaligned
\end{equation*}
For simplicity, from now on we always
omit the superscript $n$ (the time level) if the omission does not cause conflicts.
Define
\begin{equation*}
\aligned
&[d_{x}g]_{i+\frac{1}{2},j}=(g_{i+1,j}-g_{i,j})/h_x,\\
&[d_{y}g]_{i,j+\frac{1}{2}}=(g_{i,j+1}-g_{i,j})/h_y,\\
&[D_{x}g]_{i,j}=(g_{i+\frac{1}{2},j}-g_{i-\frac{1}{2},j})/h_x,\\
&[D_{y}g]_{i,j}=(g_{i,j+\frac{1}{2}}-g_{i,j-\frac{1}{2}})/h_y,\\
&[d_{t}g]^n_{i,j}=(g_{i,j}^n-g_{i,j}^{n-1})/\Delta t.
\endaligned
\end{equation*}
The following discrete-integration-by-part lemma \cite{weiser1988convergence} plays an important role in the analysis.
\begin{lemma}\label{le1}
 Let $q_{i,j},w_{1,i+1/2,j}~and~ w_{2,i,j+1/2} $ be any values such that $w_{1,1/2,j}=w_{1,N_x+1/2,j}=w_{2,i,1/2}=w_{2,i,N_y+1/2}=0$, then
$$(q,D_xw_1)_m=-(d_xq,w_1)_x,$$
$$(q,D_yw_2)_m=-(d_yq,w_2)_y.$$
\end{lemma}
  \subsubsection{SAV/CV-BCFD scheme for $H^{-1}$ gradient flow}

Let us denote by $\{Z^n, W^n, R^n\}_{n=0}^{N}$ the BCFD approximations to $\{\phi^n,\mu^n, r^n\}_{n=0}^{N}$. The scheme for $H^{-1}$ gradient flow is as follows:
for $1\leq i \leq N_x,~1\leq j\leq N_y$,
 \begin{numcases}{}
~[d_tZ]_{i,j}^{n+1}=M[D_xd_xW+D_yd_yW]_{i,j}^{n+1/2}, \label{e_H1_full-discret1}\\
~W_{i,j}^{n+1/2}=-[D_xd_xZ+D_yd_yZ]_{i,j}^{n+1/2}+\lambda Z^{n+1/2}_{i,j}\label{e_H1_full-discret2}\\ 
~+\frac{R^{n+1/2}}{\sqrt{E_1^h(\tilde{Z}^{n+1/2})}}F^{\prime}(\tilde{Z}_{i,j}^{n+1/2}),\notag
\\
~d_tR^{n+1}=\frac{1}{2\sqrt{E_1^h(\tilde{Z}^{n+1/2})}}(F^{\prime}(\tilde{Z}^{n+1/2}),
d_tZ^{n+1})_m,\label{e_H1_full-discret3}
\end{numcases}
where $\tilde{Z}^{n+1/2}$ is an approximation of $\tilde{\phi}^{n+1/2}$, and
$$E_1^h(\tilde{Z}^{n+1/2})=\sum\limits_{i=1}^{N_{x}}\sum\limits_{j=1}^{N_{y}}h_xh_yF(\tilde{Z}^{n+1/2}_{i,j}).$$

The boundary and initial approximations as follows.
\begin{equation}\label{e_H1_boundary and initial condition}
  \left\{
   \begin{array}{l}
    \displaystyle [d_xZ]_{1/2,j}^{n}=[d_xZ]_{N_x+1/2,j}^{n}=0,\quad 1\leq j\leq N_y,\\
   \displaystyle  [d_yZ]_{i,1/2}^{n}=[d_yZ]_{i,N_y+1/2}^{n}=0,\quad 1\leq i\leq N_x,\\
   \displaystyle  [d_xW]_{1/2,j}^{n}=[d_xW]_{N_x+1/2,j}^{n}=0,\quad 1\leq j\leq N_y,\\
   \displaystyle  [d_yW]_{i,1/2}^{n}=[d_yW]_{i,N_y+1/2}^{n}=0,\quad 1\leq i\leq N_x,\\
   \displaystyle Z_{i,j}^0=\phi_{0,i,j}, 1\leq i\leq N_x,1\leq j\leq N_y.
   \end{array}
   \right.
  \end{equation}
  
{\bf Remark}. {\it The solution procedure of the above scheme is described in detail in \cite{shen2018scalar,shen2017new}, and hence is omitted here.}

  \subsubsection{SAV/CV-BCFD scheme for $L^{2}$ gradient flow}

Let us denote by $\{Z^n, W^n, R^n\}_{n=0}^{N}$ the BCFD approximations to $\{\phi^n,\mu^n, r^n\}_{n=0}^{N}$.  The scheme for $L^2$ gradient flow is as follows:
for $1\leq i \leq N_x,~1\leq j\leq N_y$,
 \begin{numcases}{}
~[d_tZ]_{i,j}^{n+1}=-MW_{i,j}^{n+1/2}, \label{e_L2_full-discret1}\\
~W_{i,j}^{n+1/2}=-[D_xd_xZ+D_yd_yZ]_{i,j}^{n+1/2}+\lambda Z^{n+1/2}_{i,j}\label{e_L2_full-discret2}\\ 
~+\frac{R^{n+1/2}}{\sqrt{E_1^h(\tilde{Z}^{n+1/2})}}F^{\prime}(\tilde{Z}_{i,j}^{n+1/2}),\notag
\\
~d_tR^{n+1}=\frac{1}{2\sqrt{E_1^h(\tilde{Z}^{n+1/2})}}(F^{\prime}(\tilde{Z}^{n+1/2}),
d_tZ^{n+1})_m,\label{e_L2_full-discret3}
\end{numcases}
where $\tilde{Z}^{n+1/2}$ is an approximation of $\tilde{\phi}^{n+1/2}$. The boundary and initial conditions are given in  (\ref{e_H1_boundary and initial condition}).

  \section{Unconditional energy stability} 
We  demonstrate below that the full discrete SAV/CN-BCFD schemes are unconditionally energy stable with the discrete energy functional
\begin{equation}\label{definition of discretization of energy}
\aligned
E_d(Z^n)=\frac{\lambda}{2}\|Z^n\|^2_m+\frac{1}{2}\|\textbf{d}Z^n\|_{TM}^2+(R^n)^2,
\endaligned
\end{equation} 
where $\textbf{d}Z=(d_xZ,d_yZ)$.
  \subsection{$H^{-1}$ gradient flow}
  
 \begin{theorem}\label{thm: H1_discrete total energy}
The scheme (\ref{e_H1_full-discret1})-(\ref{e_H1_full-discret3}) is unconditionally stable and the following discrete energy law holds for any $\Delta t$:
\begin{equation}\label{e_H1_discretization of energy decay}
\aligned
\frac{1}{\Delta t}[E_d(Z^{n+1})-E_d(Z^n)]=-M\|\textbf{d}W^{n+1/2}\|_{TM}^2
,\ \ \forall n\geq 0.
\endaligned
\end{equation} 
\end{theorem}

\begin{proof}  
Multiplying equation (\ref{e_H1_full-discret1}) by $W_{i,j}^{n+1/2}h_xh_y$, and making summation on $i,j$ for $1\leq i\leq N_x,~1\leq j\leq N_y$, we have 
\begin{equation}\label{e_H1_energy analysis1}
\aligned
(d_tZ^{n+1},W^{n+1/2})_m=M(D_xd_xW^{n+1/2}+D_yd_yW^{n+1/2},W^{n+1/2})_m.
\endaligned
\end{equation} 
Using Lemma \ref{le1}, equation (\ref{e_H1_energy analysis1}) can be transformed into the following:
\begin{equation}\label{e_H1_energy analysis1_sumpart}
\aligned
(d_tZ^{n+1},W^{n+1/2})_m=&-M(\|d_xW^{n+1/2}\|_x^2+\|d_yW^{n+1/2}\|_y^2)\\
=&-M\|\textbf{d}W^{n+1/2}\|_{TM}^2.
\endaligned
\end{equation} 
Multiplying equation (\ref{e_H1_full-discret2}) by $d_tZ^{n+1}_{i,j}h_xh_y$, and making summation on $i,j$ for $1\leq i\leq N_x,~1\leq j\leq N_y$, we have 
\begin{equation}\label{e_H1_energy analysis2}
\aligned
(d_tZ^{n+1},W^{n+1/2})_m=&-(D_xd_xZ^{n+1/2}+D_yd_yZ^{n+1/2},d_tZ^{n+1})_m
\\
&+\frac{R^{n+1/2}}{\sqrt{E_1^h(\tilde{Z}^{n+1/2})}}
(F^{\prime}(\tilde{Z}^{n+1/2}), d_tZ^{n+1})_m\\
&+\lambda(Z^{n+1/2},d_tZ^{n+1})_m.
\endaligned
\end{equation} 
Using Lemma \ref{le1} again, the first term on the right hand side of equation (\ref{e_H1_energy analysis2}) can be written as:
\begin{equation}\label{e_H1_energy analysis2_first term}
\aligned
&-(D_xd_xZ^{n+1/2}+D_yd_yZ^{n+1/2},d_tZ^{n+1})_m\\
=&(d_xZ^{n+1/2},d_td_xZ^{n+1})_x+(d_yZ^{n+1/2},d_td_yZ^{n+1})_y\\
=&\frac{\|\textbf{d}Z^{n+1}\|^2_{TM}-\|\textbf{d}Z^n\|_{TM}^2}{2\Delta t}.
\endaligned
\end{equation} 
Multiplying equation (\ref{e_H1_full-discret3}) by $R^{n+1}+R^{n}$ leads to 
\begin{equation}\label{e_H1_energy analysis3}
\aligned
\frac{(R^{n+1})^2-(R^{n})^2}{\Delta t}=
\frac{R^{n+1/2}}{\sqrt{E_1^h(\tilde{Z}^{n+1/2})}}(F^{\prime}(\tilde{Z}^{n+1/2}),
d_tZ^{n+1})_M.
\endaligned
\end{equation}
Combining equation (\ref{e_H1_energy analysis3}) with equations (\ref{e_H1_energy analysis1_sumpart}) - (\ref{e_H1_energy analysis2_first term}) gives that
\begin{equation}\label{e_H1_energy analysis4}
\aligned
&\frac{(R^{n+1})^2-(R^{n})^2}{\Delta t}+\lambda \frac{\|Z^{n+1}\|^2_m-\|Z^n\|_m^2}{2\Delta t}\\
&+\frac{\|\textbf{d}Z^{n+1}\|^2_{TM}-\|\textbf{d}Z^n\|_{TM}^2}{2\Delta t}\\
=&-M\|\textbf{d}W^{n+1/2}\|_{TM}^2\leq 0,
\endaligned
\end{equation}
which implies the desired results (\ref{e_H1_discretization of energy decay}).
\quad\end{proof}

\medskip
 \subsection{$L^2$ gradient flow}
 For $L^2$ gradient flow, we shall only state the result, as its proof is essentially the same as for the $H^{-1}$ gradient flow.
 
 \begin{theorem}\label{thm: L2_discrete total energy}
The scheme (\ref{e_L2_full-discret1})-(\ref{e_L2_full-discret3}) is unconditionally stable and the following discrete energy law holds for any $\Delta t$:
\begin{equation}\label{e_L2_discretization of energy decay}
\aligned
\frac{1}{\Delta t}[E_d(Z^{n+1})-E_d(Z^n)]=-M\|W^{n+1/2}\|_m^2
,\ \ \forall n\geq 0.
\endaligned
\end{equation} 
\end{theorem}
 \section{Error estimates} 
In this section, we derive our main results of this paper, i.e.,  error estimates for the fully discrete SAV/CN-BCFD schemes.

For simplicity, we set
\begin{equation*}
\aligned
&\displaystyle e_{\phi}^n=Z^n-\phi^n,
~~\displaystyle e_{\mu}^{n}=W^{n}-\mu^n,
~~\displaystyle e_{r}^n=R^n-r^n.
\endaligned
\end{equation*} 
  \subsection{$H^{-1}$ gradient flow} 
  We shall first derive error estimates for the case of  $H^{-1}$ gradient flow.

\begin{theorem}\label{thm: H1_error_estimates} 
We assume  that $F(\phi)\in C^3(\mathbb{R})$ and $\phi\in W^{1,\infty}(J;W^{4,\infty}(\Omega)) \cap W^{3,\infty}(J;W^{1,\infty}(\Omega)),~\mu\in L^{\infty}(J;W^{4,\infty}(\Omega))$. Let $\Delta t\leq C(h_x+h_y)$, then for the discrete scheme 
(\ref{e_H1_full-discret1})-(\ref{e_H1_full-discret3}), there exists a positive constant $C$ independent of $h_x$, $h_y$ and $\Delta t$ such that
\begin{equation}\label{e_H1_error_estimate1}
\aligned
&\|Z^{k+1}-\phi^{k+1}\|_m+\|\textbf{d}Z^{k+1}-\textbf{d}\phi^{k+1}\|_{TM}+
|R^{k+1}-r^{k+1}|\\
&+\left(\sum_{n=0}^{k}\Delta t\|\textbf{d}W^{n+1/2}-\textbf{d}\mu^{n+1/2}\|_{TM}^2\right)^{1/2}\\
&+\left(\sum_{n=0}^{k}\Delta t\|W^{n+1/2}-\mu^{n+1/2}\|_{m}^2\right)^{1/2}\\
\leq&C(\|\phi\|_{W^{1,\infty}(J;W^{4,\infty}(\Omega))}+\|\mu\|_{L^{\infty}(J;W^{4,\infty}(\Omega))} )(h_x^2+h_y^2)\\
&+C\|\phi\|_{W^{3,\infty}(J;W^{1,\infty}(\Omega))}\Delta t^2.
\endaligned
\end{equation}
\end{theorem}

We shall split the proof of the above results into three lemmas below.

\medskip
\begin{lemma}\label{Proposition1}
Under the condition of Theorem \ref{thm: H1_error_estimates}, there exists a positive constant $C$ independent of $h_x$, $h_y$ and $\Delta t$ such that
\begin{equation}\label{e_H1_error_estimate19}
\aligned
&(e_r^{k+1})^2+\frac{1}{2}\|\textbf{d}e_{\phi}^{k+1}\|^2_{TM}+\frac{\lambda}{2} \|e_{\phi}^{k+1}\|_m^2+\frac{M}{2}\sum_{n=0}^{k}\Delta t\|\textbf{d}e_{\mu}^{n+1/2}\|^2_{TM}\\
\leq&C\sum_{n=0}^{k+1}\Delta t\|\textbf{d}e_{\phi}^n\|_{TM}^2
+\frac{M}{2}\sum_{n=0}^{k+1}\Delta t\|e_{\mu}^{n+1/2}\|_{m}^2\\
&+C\sum_{n=0}^{k+1}\Delta t\|e_{\phi}^n\|_m^2+C\sum_{n=0}^{k+1}\Delta t(e_r^{n})^2\\
&+C(\|\phi\|^2_{W^{1,\infty}(J;W^{4,\infty}(\Omega))}+\|\mu\|_{L^{\infty}(J;W^{4,\infty}(\Omega))}^2 )(h_x^4+h_y^4)\\
&+C\|\phi\|_{W^{3,\infty}(J;W^{1,\infty}(\Omega))}^2\Delta t^4.
\endaligned
\end{equation}
\end{lemma}
\begin{proof}
Denote
\begin{equation*}
\aligned
&\delta_x(\phi)=d_x\phi-\frac{\partial \phi}{\partial x},~\delta_y(\phi)=d_y\phi-\frac{\partial \phi}{\partial y},\\
&\delta_x(\mu)=d_x\mu-\frac{\partial \mu}{\partial x},~\delta_y(\mu)=d_y\mu-\frac{\partial \mu}{\partial y}.
\endaligned
\end{equation*}

Subtracting equation (\ref{e_true solution1}) from equation (\ref{e_H1_full-discret1}), we obtain
\begin{equation}\label{e_H1_error_estimate2}
\aligned
~[d_te_{\phi}]_{i,j}^{n+1}=&M[D_x(d_xe_{\mu}+\delta_x(\mu))+D_y(d_ye_{\mu}+\delta_y(\mu))]_{i,j}^{n+1/2}\\
&+T_{1,i,j}^{n+1/2}+T_{2,i,j}^{n+1/2},\\
\endaligned
\end{equation}
where 
\begin{equation}\label{e_error_estimate_E1}
\aligned
T_{1,i,j}^{n+1/2}=\frac{\partial \phi}{\partial t}\big|_{i,j}^{n+1/2}-[d_t\phi]_{i,j}^{n+1}\leq
C\|\phi\|_{W^{3,\infty}(J;L^{\infty}(\Omega))}\Delta t^2,
\endaligned
\end{equation}

\begin{equation}\label{e_error_estimate_E2}
\aligned
T_{2,i,j}^{n+1/2}&=M[D_x\frac{\partial \mu}{\partial x}+D_y\frac{\partial \mu}{\partial y}]_{i,j}^{n+1/2}-M\Delta \mu^{n+1/2}_{i,j}\\
&\leq CM(h_x^2+h_y^2)\|\mu\|_{L^{\infty}(J;W^{4,\infty}(\Omega))}.
\endaligned
\end{equation}
Subtracting equation (\ref{e_true solution2}) from equation (\ref{e_H1_full-discret2}) leads to
\begin{equation}\label{e_H1_error_estimate3}
\aligned
e_{\mu,i,j}^{n+1/2}=&-[D_x(d_xe_{\phi}+\delta_x(\phi))+D_y(d_ye_{\phi}+\delta_y(\phi))]_{i,j}^{n+1/2}\\
&+\lambda e_{\phi,i,j}^{n+1/2}+\frac{R^{n+1/2}}{\sqrt{E_1^h(\tilde{Z}^{n+1/2})}}F^{\prime}(\tilde{Z}_{i,j}^{n+1/2})\\
&-\frac{r^{n+1/2}}{\sqrt{E_1(\phi^{n+1/2})}}F^{\prime}(\phi_{i,j}^{n+1/2})+T_{3,i,j}^{n+1/2},
\endaligned
\end{equation}
where 
\begin{equation}\label{e_error_estimate_E3}
\aligned
T_{3,i,j}^{n+1/2}&=\Delta \phi^{n+1/2}_{i,j}-[D_x\frac{\partial \phi}{\partial x}+D_y\frac{\partial \phi}{\partial y}]_{i,j}^{n+1/2}\\
&\leq C(h_x^2+h_y^2)\|\phi\|_{L^{\infty}(J;W^{4,\infty}(\Omega))}.
\endaligned
\end{equation}
Subtracting equation (\ref{e_true solution3}) from equation (\ref{e_H1_full-discret3}) gives that
\begin{equation}\label{e_H1_error_estimate4}
\aligned
d_te_r^{n+1}=&\frac{1}{2\sqrt{E_1^h(\tilde{Z}^{n+1/2})}}(F^{\prime}(\tilde{Z}^{n+1/2}),
d_tZ^{n+1})_m\\
&-\frac{1}{2\sqrt{E_1(\phi^{n+1/2})}}\int_{\Omega}F^{\prime}(\phi^{n+1/2})\phi^{n+1/2}_t d\textbf{x}+T_{4}^{n+1/2},
\endaligned
\end{equation}
where 
\begin{equation}\label{e_error_estimate_E4}
\aligned
T_4^{n+1/2}=r_t^{n+1/2}-d_tr^{n+1}\leq
C\|r\|_{W^{3,\infty}(J)}\Delta t^2.
\endaligned
\end{equation}
Multiplying equation (\ref{e_H1_error_estimate2}) by $e_{\mu,i,j}^{n+1/2}h_xh_y$, and making summation on $i,j$ for $1\leq i\leq N_x,~1\leq j\leq N_y$, we have 
\begin{equation}\label{e_H1_error_estimate5}
\aligned
&(d_te_{\phi}^{n+1},e_{\mu}^{n+1/2})_m\\
=&M\left(D_x(d_xe_{\mu}+\delta_x(\mu))^{n+1/2}+D_y(d_ye_{\mu}+\delta_y(\mu))^{n+1/2},e_{\mu}^{n+1/2}\right)_m\\
&+(T_1^{n+1/2}, e_{\mu}^{n+1/2})_m+(T_2^{n+1/2}, e_{\mu}^{n+1/2})_m.
\endaligned
\end{equation} 
Using Lemma \ref{le1}, we can write the first term on the right hand side of equation (\ref{e_H1_error_estimate5}) as:
\begin{equation}\label{e_H1_error_estimate6}
\aligned
&M\left(D_x(d_xe_{\mu}+\delta_x(\mu))^{n+1/2}+D_y(d_ye_{\mu}+\delta_y(\mu))^{n+1/2},e_{\mu}^{n+1/2}\right)_m\\
=&-M\left((d_xe_{\mu}+\delta_x(\mu))^{n+1/2},d_xe_{\mu}^{n+1/2}\right)_x-M\left((d_ye_{\mu}+\delta_y(\mu))^{n+1/2},d_ye_{\mu}^{n+1/2}\right)_y\\
=&-M\|\textbf{d}e_{\mu}^{n+1/2}\|^2_{TM}-M(\delta_x(\mu)^{n+1/2},d_xe_{\mu}^{n+1/2})_x\\
&-M(\delta_y(\mu)^{n+1/2},d_ye_{\mu}^{n+1/2})_y.
\endaligned
\end{equation} 
Thanks to Cauchy-Schwarz inequality, the last two terms on the right hand side of equation (\ref{e_H1_error_estimate6}) can be transformed into:
\begin{equation}\label{e_H1_error_estimate6_added}
\aligned
&-M(\delta_x(\mu)^{n+1/2},d_xe_{\mu}^{n+1/2})_x-M(\delta_y(\mu)^{n+1/2},d_ye_{\mu}^{n+1/2})_y\\
\leq&\frac{M}{6}\|\textbf{d}\mu^{n+1/2}\|_{TM}^2+C\|\mu\|_{L^{\infty}(J;W^{3,\infty}(\Omega))}^2(h_x^4+h_y^4).
\endaligned
\end{equation} 

Multiplying equation (\ref{e_H1_error_estimate3}) by $d_te_{\phi,i,j}^{n+1}h_xh_y$, and making summation on $i,j$ for $1\leq i\leq N_x,~1\leq j\leq N_y$, we have 
\begin{equation}\label{e_H1_error_estimate7}
\aligned
&(e_{\mu}^{n+1/2},d_te_{\phi}^{n+1})_m=-(D_x(d_xe_{\phi}+\delta_x(\phi))^{n+1/2}+D_y(d_ye_{\phi}+\delta_y(\phi))^{n+1/2}, d_te_{\phi}^{n+1})_m\\
&+(\frac{R^{n+1/2}}{\sqrt{E_1^h(\tilde{Z}^{n+1/2})}}F^{\prime}(\tilde{Z}^{n+1/2})
-\frac{r^{n+1/2}}{\sqrt{E_1(\phi^{n+1/2})}}F^{\prime}(\phi^{n+1/2}), d_te_{\phi}^{n+1})_m\\
&+\lambda (e_{\phi}^{n+1/2},d_te_{\phi}^{n+1})_m+(T_3^{n+1/2},d_te_{\phi}^{n+1})_m.
\endaligned
\end{equation}
Similar to the estimate of equation (\ref{e_H1_energy analysis2_first term}), the 
first term on the right hand side of equation (\ref{e_H1_error_estimate7}) can be transformed into the following:
\begin{equation}\label{e_H1_error_estimate8}
\aligned
&-(D_x(d_xe_{\phi}+\delta_x(\phi))^{n+1/2}+D_y(d_ye_{\phi}+\delta_y(\phi))^{n+1/2}, d_te_{\phi}^{n+1})_m\\
=&(d_xe_{\phi}^{n+1/2},d_td_xe_{\phi}^{n+1})_x+(d_ye_{\phi}^{n+1/2},d_td_ye_{\phi}^{n+1})_y\\
&+(\delta_x(\phi)^{n+1/2},d_td_xe_{\phi}^{n+1/2})_x
+(\delta_y(\phi)^{n+1/2},d_td_ye_{\phi}^{n+1/2})_y\\
=&\frac{\|\textbf{d}e_{\phi}^{n+1}\|^2_{TM}-\|\textbf{d}e_{\phi}^n\|_{TM}^2}{2\Delta t}
+(\delta_x(\phi)^{n+1/2},d_td_xe_{\phi}^{n+1/2})_x\\
&+(\delta_y(\phi)^{n+1/2},d_td_ye_{\phi}^{n+1/2})_y.
\endaligned
\end{equation} 
The second term on the right hand side of equation (\ref{e_H1_error_estimate7}) can be
rewritten as follows:
\begin{equation}\label{e_H1_error_estimate10}
\aligned
&(\frac{R^{n+1/2}}{\sqrt{E_1^h(\tilde{Z}^{n+1/2})}}F^{\prime}(\tilde{Z}^{n+1/2})
-\frac{r^{n+1/2}}{\sqrt{E_1(\phi^{n+1/2})}}F^{\prime}(\phi^{n+1/2}), d_te_{\phi}^{n+1})_m\\
=&r^{n+1/2}(\frac{F^{\prime}(\tilde{Z}^{n+1/2})}{\sqrt{E_1^h(\tilde{Z}^{n+1/2})}}-\frac{F^{\prime}(\tilde{\phi}^{n+1/2})}{\sqrt{E_1^h(\tilde{\phi}^{n+1/2})}}, d_te_{\phi}^{n+1})_m\\
&+r^{n+1/2}(\frac{F^{\prime}(\tilde{\phi}^{n+1/2})}{\sqrt{E_1^h(\tilde{\phi}^{n+1/2})}}-\frac{F^{\prime}(\phi^{n+1/2})}{\sqrt{E_1(\phi^{n+1/2})}}, d_te_{\phi}^{n+1})_m\\
&+e_{r}^{n+1/2}(\frac{F^{\prime}(\tilde{Z}^{n+1/2})}{\sqrt{E_1^h(\tilde{Z}^{n+1/2})}},
d_te_{\phi}^{n+1})_m.
\endaligned
\end{equation}
Recalling equation (\ref{e_H1_error_estimate2}), the first term on the right hand side of equation (\ref{e_H1_error_estimate10}) can be transformed into the following:
\begin{equation}\label{e_H1_error_estimate10_added1}
\aligned
&r^{n+1/2}(\frac{F^{\prime}(\tilde{Z}^{n+1/2})}{\sqrt{E_1^h(\tilde{Z}^{n+1/2})}}-\frac{F^{\prime}(\tilde{\phi}^{n+1/2})}{\sqrt{E_1^h(\tilde{\phi}^{n+1/2})}}, d_te_{\phi}^{n+1})_m\\
=&Mr^{n+1/2}(\frac{F^{\prime}(\tilde{Z}^{n+1/2})}{\sqrt{E_1^h(\tilde{Z}^{n+1/2})}}-\frac{F^{\prime}(\tilde{\phi}^{n+1/2})}{\sqrt{E_1^h(\tilde{\phi}^{n+1/2})}}, D_x(d_xe_{\mu}+\delta_x(\mu))^{n+1/2})_m\\
&+Mr^{n+1/2}(\frac{F^{\prime}(\tilde{Z}^{n+1/2})}{\sqrt{E_1^h(\tilde{Z}^{n+1/2})}}-\frac{F^{\prime}(\tilde{\phi}^{n+1/2})}{\sqrt{E_1^h(\tilde{\phi}^{n+1/2})}}, D_y(d_ye_{\mu}+\delta_y(\mu))^{n+1/2})_m\\
&+r^{n+1/2}(\frac{F^{\prime}(\tilde{Z}^{n+1/2})}{\sqrt{E_1^h(\tilde{Z}^{n+1/2})}}-\frac{F^{\prime}(\tilde{\phi}^{n+1/2})}{\sqrt{E_1^h(\tilde{\phi}^{n+1/2})}}, T_1^{n+1/2}+T_2^{n+1/2})_m.
\endaligned
\end{equation}
Next, we shall first make the hypothesis that there exists a positive constant $C_*$ such that 
\begin{equation}\label{e_H1_error_estimate_boundedness_added1}
\aligned
\|Z^n\|_{\infty}\leq C_*.
\endaligned
\end{equation}
This hypothesis will be verified in Lemma \ref{hyperthesis} using a bootstrap argument.

Since $F(\phi)\in C^3(\mathbb{R})$, we have
\begin{equation}\label{e_H1_error_estimate10_added3}
\aligned
&\frac{d_xF^{\prime}(\tilde{Z}^{n+1/2})}{\sqrt{E_1^h(\tilde{Z}^{n+1/2})}}-\frac{d_xF^{\prime}(\tilde{\phi}^{n+1/2})}{\sqrt{E_1^h(\tilde{\phi}^{n+1/2})}}\\
=&d_xF^{\prime}(\tilde{\phi}^{n+1/2})\frac{E_1^h(\tilde{\phi}^{n+1/2})-E_1^h(\tilde{Z}^{n+1/2})}{\sqrt{E_1^h(\tilde{Z}^{n+1/2})E_1^h(\tilde{\phi}^{n+1/2})(E_1^h(\tilde{Z}^{n+1/2})+E_1^h(\tilde{\phi}^{n+1/2}))}} \\
&+\frac{d_xF^{\prime}(\tilde{Z}^{n+1/2})-d_xF^{\prime}(\tilde{\phi}^{n+1/2})}{\sqrt{E_1^h(\tilde{Z}^{n+1/2})}}.
\endaligned
\end{equation}
 Using  above and the Cauchy-Schwartz inequality, we can deduce that 
\begin{equation}\label{e_H1_error_estimate10_added2}
\aligned
&Mr^{n+1/2}(\frac{F^{\prime}(\tilde{Z}^{n+1/2})}{\sqrt{E_1^h(\tilde{Z}^{n+1/2})}}-\frac{F^{\prime}(\tilde{\phi}^{n+1/2})}{\sqrt{E_1^h(\tilde{\phi}^{n+1/2})}}, D_x(d_xe_{\mu}+\delta_x(\mu))^{n+1/2})_m\\
=&-Mr^{n+1/2}(\frac{d_xF^{\prime}(\tilde{Z}^{n+1/2})}{\sqrt{E_1^h(\tilde{Z}^{n+1/2})}}-\frac{d_xF^{\prime}(\tilde{\phi}^{n+1/2})}{\sqrt{E_1^h(\tilde{\phi}^{n+1/2})}}, (d_xe_{\mu}+\delta_x(\mu))^{n+1/2})_x\\
\leq& \frac{M}{6}\|d_xe_{\mu}^{n+1/2}\|_{x}^2+C\|r\|^2_{L^{\infty}(J)}(\|e_{\phi}^{n}\|_m^2+\|e_{\phi}^{n-1}\|_m^2)\\
&+C\|r\|^2_{L^{\infty}(J)}(\|d_xe_{\phi}^{n}\|_x^2+\|d_xe_{\phi}^{n-1}\|_x^2)\\
&+C\|\mu\|_{L^{\infty}(J;W^{3,\infty}(\Omega))}^2(h_x^4+h_y^4).
\endaligned
\end{equation}
Similarly we can obtain
\begin{equation}\label{e_H1_error_estimate10_added4}
\aligned
&Mr^{n+1/2}(\frac{F^{\prime}(\tilde{Z}^{n+1/2})}{\sqrt{E_1^h(\tilde{Z}^{n+1/2})}}-\frac{F^{\prime}(\tilde{\phi}^{n+1/2})}{\sqrt{E_1^h(\tilde{\phi}^{n+1/2})}}, D_y(d_ye_{\mu}+\delta_y(\mu))^{n+1/2})_m\\
\leq& \frac{M}{6}\|d_ye_{\mu}^{n+1/2}\|_{y}^2+C\|r\|^2_{L^{\infty}(J)}(\|e_{\phi}^{n}\|_m^2+\|e_{\phi}^{n-1}\|_m^2)\\
&+C\|r\|^2_{L^{\infty}(J)}(\|d_ye_{\phi}^{n}\|_y^2+\|d_ye_{\phi}^{n-1}\|_y^2)\\
&+C\|\mu\|_{L^{\infty}(J;W^{3,\infty}(\Omega))}^2(h_x^4+h_y^4).
\endaligned
\end{equation}
Then equation (\ref{e_H1_error_estimate10_added1}) can be estimated by:
\begin{equation}\label{e_H1_error_estimate10_added5}
\aligned
&r^{n+1/2}(\frac{F^{\prime}(\tilde{Z}^{n+1/2})}{\sqrt{E_1^h(\tilde{Z}^{n+1/2})}}-\frac{F^{\prime}(\tilde{\phi}^{n+1/2})}{\sqrt{E_1^h(\tilde{\phi}^{n+1/2})}}, d_te_{\phi}^{n+1})_m\\
\leq &\frac{M}{6}\|\textbf{d}e_{\mu}^{n+1/2}\|_{TM}^2+C\|r\|_{L^{\infty}(J)}(\|e_{\phi}^{n}\|_m^2+\|e_{\phi}^{n-1}\|_m^2)\\
&+C\|r\|_{L^{\infty}(J)}(\|\textbf{d}e_{\phi}^{n}\|_{TM}^2+\|\textbf{d}e_{\phi}^{n-1}\|_{TM}^2)
\\
&+C\|\mu\|_{L^{\infty}(J;W^{4,\infty}(\Omega))}^2(h_x^4+h_y^4)+C\|\phi\|_{W^{3,\infty}(J;L^{\infty}(\Omega))}^2\Delta t^4.
\endaligned
\end{equation}
Similar to (\ref{e_H1_error_estimate10_added1}),
the second term on the right hand side of equation (\ref{e_H1_error_estimate10}) can be controlled by:
\begin{equation}\label{e_H1_error_estimate10_added6}
\aligned
&r^{n+1/2}(\frac{F^{\prime}(\tilde{\phi}^{n+1/2})}{\sqrt{E_1^h(\tilde{\phi}^{n+1/2})}}-\frac{F^{\prime}(\phi^{n+1/2})}{\sqrt{E_1(\phi^{n+1/2})}}, d_te_{\phi}^{n+1})_m\\
\leq &\frac{M}{6}\|\textbf{d}e_{\mu}^{n+1/2}\|_{TM}^2+C\|\mu\|_{L^{\infty}(J;W^{4,\infty}(\Omega))}^2(h_x^4+h_y^4)\\
&+C\|\phi\|_{L^{\infty}(J;W^{2,\infty}(\Omega))}^2(h_x^4+h_y^4)\\
&+C\|\phi\|_{W^{3,\infty}(J;W^{1,\infty}(\Omega))}^2\Delta t^4.
\endaligned
\end{equation}
The third term on the right hand side of equation (\ref{e_H1_error_estimate7}) can be
estimated by:
\begin{equation}\label{e_H1_error_estimate9}
\aligned
&\lambda (e_{\phi}^{n+1/2},d_te_{\phi}^{n+1})_m=\lambda\frac{\|e_{\phi}^{n+1}\|_m^2-\|e_{\phi}^{n}\|_m^2 }{2\Delta t}.
\endaligned
\end{equation} 
Multiplying equation (\ref{e_H1_error_estimate4}) by $e_{r}^{n+1}+e_{r}^{n}$ leads to 
\begin{equation}\label{e_H1_error_estimate11}
\aligned
\frac{(e_r^{n+1})^2-(e_r^{n})^2}{\Delta t}=&
\frac{e_r^{n+1/2}}{\sqrt{E_1^h(\tilde{Z}^{n+1/2})}}(F^{\prime}(\tilde{Z}^{n+1/2}),
d_tZ^{n+1})_m\\
&-\frac{e_r^{n+1/2}}{\sqrt{E_1(\phi^{n+1/2})}}\int_{\Omega}F^{\prime}(\phi^{n+1/2})\phi^{n+1/2}_t d\textbf{x}\\
&+T_{4}^{n+1/2}\cdot (e_{r}^{n+1}+e_{r}^{n}).
\endaligned
\end{equation}
The first and second terms on the right hand side of equation (\ref{e_H1_error_estimate11}) can be transformed into:
\begin{equation}\label{e_H1_error_estimate13}
\aligned
&\frac{e_r^{n+1/2}}{\sqrt{E_1^h(\tilde{Z}^{n+1/2})}}(F^{\prime}(\tilde{Z}^{n+1/2}),
d_tZ^{n+1})_m-\frac{e_r^{n+1/2}}{\sqrt{E_1(\phi^{n+1/2})}}\int_{\Omega}F^{\prime}(\phi^{n+1/2})\phi^{n+1/2}_t d\textbf{x}\\
=&\frac{e_{r}^{n+1/2}}{\sqrt{E_1(\phi^{n+1/2})}}\left((F^{\prime}(\phi^{n+1/2}),d_t\phi^{n+1})_m-\int_{\Omega}F^{\prime}(\phi^{n+1/2})\phi^{n+1/2}_t d\textbf{x}\right)\\
&+\frac{e_{r}^{n+1/2}}{\sqrt{E_1^h(\tilde{Z}^{n+1/2})}}(F^{\prime}(\tilde{Z}^{n+1/2}),
d_te_{\phi}^{n+1})_m\\
&+e_{r}^{n+1/2}(\frac{F^{\prime}(\tilde{Z}^{n+1/2})}{\sqrt{E_1^h(\tilde{Z}^{n+1/2})}}-
\frac{F^{\prime}(\phi^{n+1/2})}{\sqrt{E_1(\phi^{n+1/2})}},d_t\phi^{n+1})_m.
\endaligned
\end{equation}
Since $F(\phi)\in C^3(\mathbb{R})$, we have that
\begin{equation}\label{e_H1_error_estimate14}
\aligned
&e_{r}^{n+1/2}(\frac{F^{\prime}(\tilde{Z}^{n+1/2})}{\sqrt{E_1^h(\tilde{Z}^{n+1/2})}}-
\frac{F^{\prime}(\phi^{n+1/2})}{\sqrt{E_1(\phi^{n+1/2})}},d_t\phi^{n+1})_m\\
=&e_{r}^{n+1/2}(\frac{F^{\prime}(\tilde{Z}^{n+1/2})}{\sqrt{E_1^h(\tilde{Z}^{n+1/2})}}-
\frac{F^{\prime}(\phi^{n+1/2})}{\sqrt{E_1^h(\tilde{Z}^{n+1/2})}},d_t\phi^{n+1})_m\\
&+e_{r}^{n+1/2}(\frac{F^{\prime}(\phi^{n+1/2})}{\sqrt{E_1^h(\tilde{Z}^{n+1/2})}}-
\frac{F^{\prime}(\phi^{n+1/2})}{\sqrt{E_1(\phi^{n+1/2})}},d_t\phi^{n+1})_m\\
\leq &C(e_r^{n+1/2})^2+C\|\phi\|^2_{W^{1,\infty}(J;L^{\infty}(\Omega))}(\|e_{\phi}^{n}\|_m^2+\|e_{\phi}^{n-1}\|_m^2).
\endaligned
\end{equation}
Recalling the midpoint approximation property of the rectangle quadrature formula, we can obtain that 
\begin{equation}\label{e_H1_error_estimate15}
\aligned
&\frac{e_{r}^{n+1/2}}{\sqrt{E_1(\phi^{n+1/2})}}\left((F^{\prime}(\phi^{n+1/2}),d_t\phi^{n+1})_m-\int_{\Omega}F^{\prime}(\phi^{n+1/2})\phi^{n+1/2}_t d\textbf{x}\right)\\
\leq &C(e_r^{n+1/2})^2+
C\|\phi\|^2_{W^{1,\infty}(J;W^{2,\infty}(\Omega))}(h_x^4+h_y^4).
\endaligned
\end{equation}
Combining equation (\ref{e_H1_error_estimate11}) with equations (\ref{e_H1_error_estimate5})-(\ref{e_H1_error_estimate15}) and using Cauchy-Schwarz
inequality result in
\begin{eqnarray}
&&\frac{(e_r^{n+1})^2-(e_r^{n})^2}{\Delta t}+\frac{\|\textbf{d}e_{\phi}^{n+1}\|^2_{TM}-\|\textbf{d}e_{\phi}^n\|_{TM}^2}{2\Delta t} \nonumber\\
&&+\lambda\frac{\|e_{\phi}^{n+1}\|_m^2-\|e_{\phi}^{n}\|_m^2 }{2\Delta t}+M\|\textbf{d}e_{\mu}^{n+1/2}\|^2_{TM}\label{e_H1_error_estimate12}\\
&\leq&\frac{M}{2}\|\textbf{d}e_{\mu}^{n+1/2}\|_{TM}^2+C\|r\|^2_{L^{\infty}(J)}(\|e_{\phi}^{n}\|_m^2+\|e_{\phi}^{n-1}\|_m^2) \nonumber\\
&&+C\|r\|^2_{L^{\infty}(J)}(\|\textbf{d}e_{\phi}^{n}\|_{TM}^2+\|\textbf{d}e_{\phi}^{n-1}\|_{TM}^2)
\nonumber\\
&&-(\delta_x(\phi)^{n+1/2},d_td_xe_{\phi}^{n+1/2})_x
-(\delta_y(\phi)^{n+1/2},d_td_ye_{\phi}^{n+1/2})_y\nonumber\\
&&+(T_3^{n+1/2},d_te_{\phi}^{n+1})_m-(T_1^{n+1/2}, e_{\mu}^{n+1/2})_m\nonumber\\
&&-(T_2^{n+1/2}, e_{\mu}^{n+1/2})_m+T_{4}^{n+1/2}\cdot (e_{r}^{n+1}+e_{r}^{n})\nonumber\\
&&+C(e_r^{n+1/2})^2+C\|\phi\|^2_{W^{1,\infty}(J;L^{\infty}(\Omega))}(\|e_{\phi}^{n}\|_m^2+\|e_{\phi}^{n-1}\|_m^2)\nonumber\\
&&+C(\|\phi\|^2_{W^{1,\infty}(J;W^{2,\infty}(\Omega))}+\|\mu\|_{L^{\infty}(J;W^{4,\infty}(\Omega))}^2 )(h_x^4+h_y^4)\nonumber\\
&&+C\|\phi\|_{W^{3,\infty}(J;W^{1,\infty}(\Omega))}^2\Delta t^4.
\end{eqnarray}
From the discrete-integration-by-parts, 
\begin{equation}\label{e_H1_error_estimate16}
\aligned
\sum_{n=0}^{k}\Delta t&(f^n,d_tg^{n+1})
=-\sum_{n=1}^{k}\Delta t(d_tf^n,g^n)\\
&+(f^k,g^{k+1})+(f^0,g^0).
\endaligned
\end{equation}
we find 
\begin{equation}\label{e_H1_error_estimate17}
\aligned
&\sum_{n=0}^{k}\Delta t(T_3^{n+1/2},d_te_{\phi}^{n+1})\\
=&-\sum_{n=1}^{k}\Delta t(d_tT_3^{n+1/2},e_{\phi}^n)
+(T_3^{k+1/2},e_{\phi}^{k+1})+(T_3^{1/2},e_{\phi}^0)\\
\leq& C\sum_{n=1}^{k}\Delta t\|e_{\phi}^n\|_m^2+\frac{\lambda}{4}\|e_{\phi}^{k+1}\|_m^2+C\|\phi\|^2_{W^{1,\infty}(J;W^{4,\infty}(\Omega))}(h_x^4+h_y^4).
\endaligned
\end{equation}
Similarly we have
\begin{equation}\label{e_H1_error_estimate18}
\aligned
&-\sum_{n=0}^{k}\Delta t(\delta_x(\phi)^{n+1/2},d_td_xe_{\phi}^{n+1/2})_x
-\sum_{n=0}^{k}\Delta t(\delta_y(\phi)^{n+1/2},d_td_ye_{\phi}^{n+1/2})_y\\
\leq&C\sum_{n=1}^{k}\Delta t\|\textbf{d}e_{\phi}^n\|_{TM}^2+\frac{\lambda}{4}\|e_{\phi}^{k+1}\|_m^2+C\|\phi\|^2_{W^{1,\infty}(J;W^{3,\infty}(\Omega))}(h_x^4+h_y^4).
\endaligned
\end{equation}
Multiplying equation (\ref{e_H1_error_estimate12}) by $\Delta t$, summing over
$n,~n=0,1,\ldots,k$ and combining with equations (\ref{e_H1_error_estimate17}) and 
(\ref{e_H1_error_estimate18}), we can obtain (\ref{e_H1_error_estimate19}). 
\end{proof}
\medskip

\begin{lemma}\label{Proposition2}
Under the condition of Theorem \ref{thm: H1_error_estimates}, there exists a positive constant $C$ independent of $h_x$, $h_y$ and $\Delta t$ such that
\begin{equation}\label{e_H1_error_estimate25}
\aligned
&\|e_{\phi}^{k+1}\|_m^2+M\sum\limits_{n=0}^{k}\Delta t\|e_{\mu}^{n+1/2}\|_m^2\\
\leq& C\sum\limits_{n=0}^{k}\Delta t(e_r^{n+1})^2+C\sum\limits_{n=0}^{k}\Delta t\|e_{\phi}^n\|_m^2\\
&+\frac{M}{4}\sum\limits_{n=0}^{k}\Delta t\|\textbf{d}e_{\mu}^{n+1/2}\|_{TM}^2
+C\sum\limits_{n=0}^{k}\Delta t\|\textbf{d}e_{\phi}^{n+1/2}\|_{TM}^2\\
&+C(\|\mu\|_{L^{\infty}(J;W^{4,\infty}(\Omega))}^2+\|\phi\|_{L^{\infty}(J;W^{4,\infty}(\Omega))}^2)(h_x^4+h_y^4)\\
&+C\|\phi\|_{W^{3,\infty}(J;L^{\infty}(\Omega))}^2\Delta t^4.
\endaligned
\end{equation} 
\end{lemma}

\begin{proof} 
Multiplying equation (\ref{e_H1_error_estimate2}) by $e_{\phi,i,j}^{n+1/2}h_xh_y$, and making summation on $i,j$ for $1\leq i\leq N_x,~1\leq j\leq N_y$, we have 
\begin{equation}\label{e_H1_error_estimate20}
\aligned
&(d_te_{\phi}^{n+1},e_{\phi}^{n+1/2})_m\\
=&M\left(D_x(d_xe_{\mu}+\delta_x(\mu))^{n+1/2}+D_y(d_ye_{\mu}+\delta_y(\mu))^{n+1/2},e_{\phi}^{n+1/2}\right)_m\\
&+(T_1^{n+1/2}, e_{\phi}^{n+1/2})_m+(T_2^{n+1/2}, e_{\phi}^{n+1/2})_m.
\endaligned
\end{equation} 
Using Lemma \ref{le1}, the first term on the right hand side of equation (\ref{e_H1_error_estimate20}) can be transformed into the following:
\begin{equation}\label{e_H1_error_estimate21}
\aligned
&M\left(D_x(d_xe_{\mu}+\delta_x(\mu))^{n+1/2}+D_y(d_ye_{\mu}+\delta_y(\mu))^{n+1/2},e_{\phi}^{n+1/2}\right)_m\\
=&-M\left((d_xe_{\mu}+\delta_x(\mu))^{n+1/2}, d_xe_{\phi}^{n+1/2}\right)_x\\
&-M\left((d_ye_{\mu}+\delta_y(\mu))^{n+1/2}, d_ye_{\phi}^{n+1/2}\right)_y.
\endaligned
\end{equation} 
The first term on the right hand side of equation (\ref{e_H1_error_estimate21}) can be estimated as:
\begin{equation}\label{e_H1_error_estimate22}
\aligned
&-M\left((d_xe_{\mu}+\delta_x(\mu))^{n+1/2}, d_xe_{\phi}^{n+1/2}\right)_x\\
=&-M\left(d_xe_{\mu}^{n+1/2}, (d_xe_{\phi}+\delta_x(\phi))^{n+1/2}\right)_x\\
&+M(d_xe_{\mu}^{n+1/2}, \delta_x(\phi)^{n+1/2})_x
-M(\delta_x(\mu)^{n+1/2}, d_xe_{\phi}^{n+1/2})_x\\
\leq&M\left(e_{\mu}^{n+1/2},D_x(d_xe_{\phi}+\delta_x(\phi))^{n+1/2} \right)_m\\
&+\frac{M}{4}\|d_xe_{\mu}^{n+1/2}\|_x^2+C\|d_xe_{\phi}^{n+1/2}\|_x^2\\
&+C(\|\mu\|^2_{L^{\infty}(J;W^{3,\infty}(\Omega))}+\|\phi\|^2_{L^{\infty}(J;W^{3,\infty}(\Omega))})(h_x^4+h_y^4).
\endaligned
\end{equation} 
In the $y$ direction, we have the similar estimates. Then the left hand side in (\ref{e_H1_error_estimate21}) can be bounded by:
\begin{equation}\label{e_H1_error_estimate23}
\aligned
&M\left(D_x(d_xe_{\mu}+\delta_x(\mu))^{n+1/2}+D_y(d_ye_{\mu}+\delta_y(\mu))^{n+1/2},e_{\phi}^{n+1/2}\right)_m\\
\leq &M\left(e_{\mu}^{n+1/2},D_x(d_xe_{\phi}+\delta_x(\phi))^{n+1/2}+D_y(d_ye_{\phi}+\delta_y(\phi))^{n+1/2} \right)_m\\
&+\frac{M}{4}\|\textbf{d}e_{\mu}^{n+1/2}\|_{TM}^2+C\|\textbf{d}e_{\phi}^{n+1/2}\|_{TM}^2\\
&+C(\|\mu\|_{L^{\infty}(J;W^{3,\infty}(\Omega))}^2+\|\phi\|_{L^{\infty}(J;W^{3,\infty}(\Omega))}^2)(h_x^4+h_y^4).
\endaligned
\end{equation} 
Thanks to (\ref{e_H1_error_estimate3}) and (\ref{e_H1_error_estimate10}), the first term on the right hand side of  (\ref{e_H1_error_estimate23}) can be estimated as follows:
\begin{equation}\label{e_H1_error_estimate24}
\aligned
&M\left(e_{\mu}^{n+1/2},D_x(d_xe_{\phi}+\delta_x(\phi))^{n+1/2}+D_y(d_ye_{\phi}+\delta_y(\phi))^{n+1/2} \right)_m\\
=&M\left(e_{\mu}^{n+1/2},\frac{R^{n+1/2}}{\sqrt{E_1^h(\tilde{Z}^{n+1/2})}}F^{\prime}(\tilde{Z}^{n+1/2})-\frac{r^{n+1/2}}{\sqrt{E_1(\phi^{n+1/2})}}F^{\prime}(\phi^{n+1/2}) \right)_m\\
&+M(e_{\mu}^{n+1/2},\lambda e_{\phi}^{n+1/2})_m+M(e_{\mu}^{n+1/2},T_3^{n+1/2})_m
-M\|e_{\mu}^{n+1/2}\|_m^2\\
\leq& \frac{M}{2}\|e_{\mu}^{n+1/2}\|_m^2+C(e_r^{n+1}+e_r^{n})^2+
C(\|e_{\phi}^n\|_m^2+\|e_{\phi}^{n-1}\|_m^2)\\
&-M\|e_{\mu}^{n+1/2}\|_m^2+C\|\phi\|_{L^{\infty}(J;W^{4,\infty}(\Omega))}^2(h_x^4+h_y^4).
\endaligned
\end{equation}
Combining equation (\ref{e_H1_error_estimate20}) with equations (\ref{e_H1_error_estimate23}) and (\ref{e_H1_error_estimate24}) and multiplying equation (\ref{e_H1_error_estimate12}) by $2\Delta t$, summing over
$n,~n=0,1,\ldots,k$ lead to (\ref{e_H1_error_estimate25}).
\end{proof}

\begin{lemma}\label{hyperthesis}
 Under the condition of Theorem \ref{thm: H1_error_estimates}, there exists a positive constant $C_*$ independent of $h_x$, $h_y$ and $\Delta t$ such that
 $$\|Z^n\|_{\infty}\le C_* \;\text{ for  all }\; n.$$
\end{lemma}
\begin{proof} We proceed in two steps.

\textit{\textbf{Step 1}} (Definition of $C_*$):  
Using the scheme (\ref{e_H1_full-discret1})-(\ref{e_H1_full-discret3}) for $n=0$ and applying the inverse assumption, we can get the approximation $Z^1$ with the following property:
\begin{equation*}
\aligned
\|Z^1\|_{\infty}\leq&\|Z^1-\phi^1\|_{\infty}+\|\phi^1\|_{\infty}
\leq \|Z^1-\Pi_h\phi^1\|_{\infty}+\|\Pi_h\phi^1-\phi^1\|_{\infty}+\|\phi^1\|_{\infty}\\
\leq&Ch^{-1}(\|Z^1-\phi^1 \|_m+\|\phi^1-\Pi_h\phi^1\|_m)+\|\Pi_h\phi^1-\phi^1\|_{\infty}+\|\phi^1\|_{\infty}\\
\leq&C(h+h^{-1}\Delta t^2)+\|\phi^1\|_{\infty}\leq C.
\endaligned
\end{equation*}
where $h=\max\{h_x,h_y\}$ and $\Pi_h$ is an bilinear interpolant operator with the following estimate \cite{dawson1998two}:
\begin{equation}\label{e_H1_error_estimate_boundedness_added2}
\aligned
\|\Pi_h\phi^1-\phi^1\|_{\infty}\leq Ch^2.
\endaligned
\end{equation}
 Thus we can choose the positive constant $C_*$ independent of $h$ and $\Delta t$ such that
\begin{align*}
C_*&\geq \max\{\|Z^{1}\|_{\infty}, 2\|\phi^{n}\|_{\infty}\}.
\end{align*}

\textit{\textbf{Step 2}} (Induction): By the definition of $C_*$, it is trivial that hypothesis (\ref{e_H1_error_estimate_boundedness_added1}) holds true for $l=1$. Supposing that $\|Z^{l-1}\|_{\infty}\leq C_*$ holds true for an integer $l=1,\cdots,k+1$, with the aid of the estimate (\ref{e_H1_error_estimate27}), we have that
$$\|Z^{l}-\phi^l\|_m\leq C(\Delta t^2+h^2).$$
Next we prove that $\|Z^{l}\|_{\infty}\leq C_*$ holds true.
Since
\begin{equation}\label{e_H1_error_estimate_boundedness_added3}
\aligned
\|Z^l\|_{\infty}\leq&\|Z^l-\phi^l\|_{\infty}+\|\phi^l\|_{\infty}
\leq \|Z^l-\Pi_h\phi^l\|_{\infty}+\|\Pi_h\phi^l-\phi^l\|_{\infty}+\|\phi^l\|_{\infty}\\
\leq&Ch^{-1}(\|Z^l-\phi^l \|_m+\|\phi^l-\Pi_h\phi^l\|_m)+\|\Pi_h\phi^l-\phi^l\|_{\infty}+\|\phi^l\|_{\infty}\\
\leq&C_1(h+h^{-1}\Delta t^2)+\|\phi^1\|_{\infty}.
\endaligned
\end{equation}
Let $\Delta t\leq C_2h$ and a positive constant $h_1$ be small enough to satisfy
$$C_1(1+C_2^2)h_1\leq\frac{C_*}{2}.$$
Then for $h\in (0,h_1],$ we derive from (\ref{e_H1_error_estimate_boundedness_added3}) that
\begin{equation*}
\aligned
\|Z^l\|_{\infty}
\leq&C_1(h+h^{-1}\Delta t^2)+\|\phi^l\|_{\infty}\\
\leq &C_1(h_1+C_2^2h_1)+\frac{C_*}{2}
\leq C_*.
\endaligned
\end{equation*}
This completes the induction.      
\end{proof}

We are now in position  to prove our main results.

\begin{proof}[Proof of Theorem 4]

Thanks to the above three lemmas, we can obtain 
\begin{equation}\label{e_H1_error_estimate26}
\aligned
&(e_r^{k+1})^2+\frac{1}{2}\|\textbf{d}e_{\phi}^{k+1}\|^2_{TM}+ \|e_{\phi}^{k+1}\|_m^2\\
&+\frac{M}{4}\sum_{n=0}^{k}\Delta t\|\textbf{d}e_{\mu}^{n+1/2}\|^2_{TM}+
\frac{M}{2}\sum\limits_{n=0}^{k}\Delta t\|e_{\mu}^{n+1/2}\|_m^2
\\
\leq&C\sum_{n=0}^{k+1}\Delta t\|\textbf{d}e_{\phi}^n\|_{TM}^2+C\sum_{n=0}^{k+1}\Delta t\|e_{\phi}^n\|_m^2+C\sum_{n=0}^{k+1}\Delta t(e_r^{n})^2\\
&+C(\|\phi\|^2_{W^{1,\infty}(J;W^{4,\infty}(\Omega))}+\|\mu\|_{L^{\infty}(J;W^{4,\infty}(\Omega))}^2 )(h_x^4+h_y^4)\\
&+C\|\phi\|_{W^{3,\infty}(J;W^{1,\infty}(\Omega))}^2\Delta t^4.
\endaligned
\end{equation}
Finally applying the discrete Gronwall's inequality, we arrive at the desired result:
\begin{equation}\label{e_H1_error_estimate27}
\aligned
&(e_r^{k+1})^2+\|\textbf{d}e_{\phi}^{k+1}\|^2_{TM}+ \|e_{\phi}^{k+1}\|_m^2\\
&+\sum_{n=0}^{k}\Delta t\|\textbf{d}e_{\mu}^{n+1/2}\|^2_{TM}+
\sum\limits_{n=0}^{k}\Delta t\|e_{\mu}^{n+1/2}\|_m^2
\\
\leq&C(\|\phi\|^2_{W^{1,\infty}(J;W^{4,\infty}(\Omega))}+\|\mu\|_{L^{\infty}(J;W^{4,\infty}(\Omega))}^2 )(h_x^4+h_y^4)\\
&+C\|\phi\|_{W^{3,\infty}(J;W^{1,\infty}(\Omega))}^2\Delta t^4.
\endaligned
\end{equation}   
Thus, the proof of Theorem 4 is complete.

\end{proof}

 \subsection{$L^2$ gradient flow}
 For the $L^2$ gradient flow, we shall only state the error estimates below, as their proofs are
essentially the same as for the $H^{-1}$ gradient flow.

\begin{theorem}\label{thm: L2_error_estimates}
We  assume  that $F(\phi)\in C^3(\mathbb{R})$ and $\phi\in W^{1,\infty}(J;W^{4,\infty}(\Omega)) \cap W^{3,\infty}(J;W^{1,\infty}(\Omega))$ and $\Delta t\leq C(h_x+h_y)$. Then for the discrete scheme 
(\ref{e_L2_full-discret1})-(\ref{e_L2_full-discret3}), there exists a positive constant $C$ independent of $h_x$, $h_y$ and $\Delta t$ such that
\begin{equation}\label{e_L2_error_estimate1}
\aligned
&\|Z^{k+1}-\phi^{k+1}\|_m+\|\textbf{d}Z^{k+1}-\textbf{d}\phi^{k+1}\|_{TM}+
|R^{k+1}-r^{k+1}|\\
\leq&C\|\phi\|_{W^{3,\infty}(J;W^{1,\infty}(\Omega))}\Delta t^2+C\|\phi\|_{W^{1,\infty}(J;W^{4,\infty}(\Omega))}(h_x^2+h_y^2).
\endaligned
\end{equation}
\end{theorem}

\medskip
\section{Numerical simulations}
We present in this section various numerical experiments to verify the energy stability and accuracy of the proposed numerical schemes.

\subsection{Accuracy test for Allen-Cahn and Cahn-Hilliard equations}
We consider the free energy
\begin{equation}\label{e_numerical_1}
\aligned
E(\phi)=\int_{\Omega}\left(\frac{1}{2}|\nabla \phi |^2+\frac{1}{4\epsilon^2}(\phi^2-1)^2 \right)d\textbf{x}.
\endaligned
\end{equation}
and for better accuracy, rewrite  it as
\begin{equation}\label{e_numerical_1_transformed}
\aligned
E(\phi)=\int_{\Omega}\left(\frac{1}{2}|\nabla \phi |^2+\frac{\beta}{2\epsilon^2}\phi^2
+\frac{1}{4\epsilon^2}(\phi^2-1-\beta)^2 -\frac{\beta^2+2\beta}{4\epsilon^2}\right)d\textbf{x},
\endaligned
\end{equation}
where $\beta$ is a positive number to be chosen.
To apply our schemes (\ref{e_H1_full-discret1})-(\ref{e_H1_full-discret3}) or 
(\ref{e_L2_full-discret1})-(\ref{e_L2_full-discret3}) to the system (\ref{e_model}), we 
drop the constant in the free energy and specify the operator $\mathcal{G}$, the energy $E_1(\phi)$ and $\lambda$ as follows:
\begin{equation}\label{e_numerical_2}
\aligned
\mathcal{G}=-(-\Delta)^s,~~E_1(\phi)=\frac{1}{4\epsilon^2}\int_{\Omega}(\phi^2-1-\beta)^2d\textbf{x},~~\lambda=\frac{\beta}{\epsilon^2}.
\endaligned
\end{equation}
 The system (\ref{e_model}) becomes the standard Allen-Cahn equation with $s=0$, and the standard Cahn-Hilliard equation with $s=1$.

We denote
\begin{flalign*}
\renewcommand{\arraystretch}{1.5}
  \left\{
   \begin{array}{l}
\|f-g\|_{\infty,2}=\max\limits_{0\leq n\leq k}\left\{\|f^{n+q}-g^{n+q}\|_X\right\},\\
\|f-g\|_{2,2}=\left(\sum\limits_{n=0}^{k}\Delta t\left\|f^{n+q}-g^{n+q}\right\|_{X}^2\right)^{1/2},\\
\|R-r\|_{\infty}=\max\limits_{0\leq n\leq k}\{R^{n+1}-r^{n+1}\},\\
\end{array}\right.
\end{flalign*}
where $q=\frac{1}{2},~1$ and $X=m,~TM$.

In the following simulations, we choose $ \Omega=(0,1)\times(0,1)$  and $C_0=0$.

\subsubsection{Convergence rates of the SAV/CN-BCFD scheme for Allen-Cahn equation}

\noindent{\bf Example 1}. We take $T=0.5,~~\mathcal{G}=-1,~~\beta=0, ~~M=0.01,~\epsilon=0.08,\ \Delta t=5E-4$, and the initial solution $\phi_0=\cos(\pi x)\cos(\pi y)$.
To get around the fact that we do not have possession of exact solution, we measure
Cauchy error, which is similar to \cite{chen2016efficient,Wise2007Solving,Diegel2013Analysis}. Specifically, the error between two different grid spacings $h$ and $\frac{h}{2}$ is calculated by $\|e_{\zeta}\|=\|\zeta_h-\zeta_{h/2}\|$.

 The numerical results are listed in Table \ref{table1_example1}. we observe  the second-order convergence predicted by the error estimates in Theorem \ref{thm: L2_error_estimates}.

\begin{table}[htbp]
\renewcommand{\arraystretch}{1.1}
\small
\centering
\caption{Errors and convergence rates of Example 1.}\label{table1_example1}
\begin{tabular}{p{1cm}p{1.5cm}p{1cm}p{1.5cm}p{1cm}p{1.5cm}p{1cm}}\hline
$h$    &$\|e_Z\|_{\infty,2}$    &Rate &$\|e_{\textbf{d}Z}\|_{\infty,2}$   &Rate  
&$\|e_{W}\|_{\infty}$    &Rate   \\ \hline
$1/10$     &6.36E-3                & ---    &5.96E-2         &---  &5.93E-3         &---\\
$1/20$      &1.59E-3                & 2.00    &1.57E-2         &1.93 &1.47E-3         &2.01\\
$1/40$      &3.98E-4                &2.00     &3.98E-3         &1.98 &3.69E-4         &2.00\\
$1/80$      &9.96E-5                &2.00    &9.98E-4         &1.99   &9.23E-5         &2.00\\
\hline
\end{tabular}
\end{table}


\subsubsection{Convergence rates of SAV/CN-BCFD scheme for Cahn-Hilliard equation}

\noindent{\bf Example 2}. We take $T=0.5,~~\mathcal{G}=\Delta,~~\beta=0, ~~M=0.01,~\epsilon=0.2,\ \Delta t=5E-4$, with  the same initial solution as in Example 1.
 The numerical results are listed in Tables \ref{table1_example2} and \ref{table2_example2}. Again, we observe the  expected second-order convergence rate in various discrete norms.

 \begin{table}[htbp]
\renewcommand{\arraystretch}{1.1}
\small
\centering
\caption{Errors and convergence rates of example 2.}\label{table1_example2}
\begin{tabular}{p{1cm}p{1.5cm}p{0.7cm}p{1.8cm}p{0.7cm}p{1.8cm}p{0.7cm}}\hline
$h$    &$\|e_{Z}\|_{\infty,2}$    &Rate &$\|e_{\textbf{d}Z}\|_{\infty,2}$   &Rate  
&$\|e_{R}\|_{\infty}$    &Rate   \\ \hline
$1/10$     &5.49E-3                & ---    &2.78E-2         &---  &4.88E-3         &---\\
$1/20$      &1.36E-3                & 2.01    &6.91E-3         &2.01 &1.20E-3         &2.02\\
$1/40$      &3.41E-4                &2.00     &1.73E-3         &2.00 &3.00E-4         &2.00\\
$1/80$      &8.51E-5                &2.00    &4.31E-4         &2.00   &7.49E-5         &2.00\\
\hline
\end{tabular}
\end{table}

\begin{table}[htbp]
\renewcommand{\arraystretch}{1.1}
\small
\centering
\caption{{Errors and convergence rates of example 2.}}\label{table2_example2}
\begin{tabular}{p{1cm}p{2cm}p{1cm}p{2cm}p{1cm}}\hline
$h$    &$\|e_{W}\|_{2,2}$    &Rate &$\|e_{\textbf{d}W}\|_{2,2}$   &Rate  
 \\ \hline
$1/10$     &2.50E-2                & ---    &2.18E-1         &---  \\
$1/20$      &6.11E-3                & 2.03    &5.46E-2         &2.00 \\
$1/40$      &1.52E-3                &2.01     &1.37E-2         &2.00 \\
$1/80$      &3.79E-4                &2.00    &3.42E-3         &2.00   \\
\hline
\end{tabular}
\end{table}

\subsection{Coarsening dynamics and adaptive time stepping}
In this example, we simulate the coarsening dynamics of the Cahn-Hilliard equation. 

Since the scheme (\ref{e_H1_full-discret1})-(\ref{e_H1_full-discret3}) is unconditionally energy stable, we can choose time steps according to accuracy only with an adaptive time stepping.  Actually in many situations, the energy and solution of gradient flows can vary drastically in certain time intervals, but only slightly elsewhere. In order to maintain the desired accuracy, we adjust the time sizes based on an adaptive 
time-stepping strategy  below (Ref. \cite{Gomez2011Provably,shen2017new}).  
\begin{algorithm}
\caption{Adaptive time stepping procedure}
\ \textbf{Given:} $\textbf{Z}^{n}$ and $\Delta t^n$.
\begin{algorithmic}[1]
\STATE Computer $\textbf{Z}_{Ref}^{n+1}$ using a first order SAV-BCFD scheme and $\Delta t^n$.\\
\STATE Computer $\textbf{Z}^{n+1}$ using the SAV/CN-BCFD scheme (\ref{e_H1_full-discret1})-(\ref{e_H1_full-discret3}) and $\Delta t^n$.\\
\STATE Calculate $e^{n+1}=\|\textbf{Z}_{Ref}^{n+1}-\textbf{Z}^{n+1}\|/\|\textbf{Z}^{n+1}\|$.\\
\STATE \textbf{If} $e^{n+1}>tol$ \textbf{then}\\
\quad Recalculate time step $\Delta t^{n}\leftarrow \max\{\Delta t_{min}, \min\{A_{dp}(e^{n+1},\Delta t^n),\Delta t_{max}\}\}$.
\STATE \quad \textbf{goto} 1\\
\STATE \textbf{else}\\
\quad Update time step $\Delta t^{n+1}\leftarrow \max\{\Delta t_{min}, \min\{A_{dp}(e^{n+1},\Delta t^n),\Delta t_{max}\}\}$.
\STATE \textbf{endif}\\
\end{algorithmic}
\end{algorithm}
We update the time step size by using the formula
\begin{equation}\label{update time step formula}
\aligned
A_{dp}(e,\Delta t)=\rho\left( \frac{tol}{e}\right)^{1/2}\Delta t,
\endaligned
\end{equation}
where $\rho$ is a default safety coefficient, $tol$ is a reference tolerance, and $e$ is the relative error at each time level. In this simulation, we take 
\begin{flalign*}
\renewcommand{\arraystretch}{1.5}
  \left\{
   \begin{array}{l}
  \mathcal{G}=\Delta,~\Delta t_{max}=10^{-2},~\Delta t_{min}=10^{-5},\ tol=10^{-3},\\
  M=0.002,~\epsilon=0.01,~~\beta=6,\;
  \rho=0.9,
\end{array}\right.
\end{flalign*}
with a  random initial condition with values in  $[-0.05,0.05]$, and the initial time step is taken as $\Delta t_{min}$.



To demonstrate the effectivity of the SAC/CN-BCFD scheme with adaptive time stepping, we 
compute the reference solutions with a small uniform time step $\Delta t=10^{-5}$ and a large uniform time step $\Delta t=10^{-3}$ respectively. Characteristic evolutions of the phase functions are presented in Fig. \ref{fig2_example5}. We also present in Fig. \ref{fig3_example5} the energy evolutions and the roughness of interface, where the roughness measure function $R(t)$ is defined as follows:
\begin{equation}\label{e_example 5_Roughness}
\aligned
R(t)=\sqrt{\frac{1}{|\Omega|}\int_{\Omega}(\phi-\bar{\phi})^2d\Omega},
\endaligned
\end{equation}
with $\bar{\phi}=\frac{1}{|\Omega|}\int_{\Omega}\phi d\Omega$. One observes that the solution obtained with adaptive time steps  is  consistent with the reference solution obtained with a small time step, while the solution with large time step deviates from the reference solution.  This is also verified by both the energy evolutions and roughness measure function $R(t)$. We present in Fig. \ref{fig4_revised_different_epsilon} the adaptive time steps for different $\epsilon=0.02, \ 0.01,\ 0.005$. We  observe that there are about two-orders of magnitude variation in the time steps with the adaptive time stepping, which indicates that the adaptive time stepping for the SAV/CN-BCFD scheme is very efficient. 

\begin{figure}
\centering
\begin{minipage}[c]{0.14\linewidth}
$\Delta t=10^{-5}$
\end{minipage}
\subfigure[$T=0.02$]{
    \begin{minipage}[c]{0.26\textwidth}
    \includegraphics[width=1\textwidth]{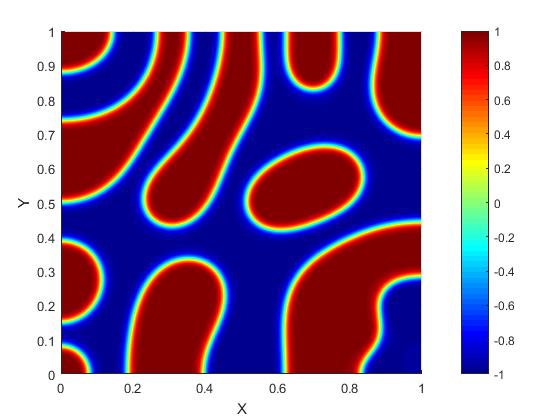}
    \end{minipage}
}
\subfigure[T=0.10]{
    \begin{minipage}[c]{0.26\textwidth}
    \includegraphics[width=1\textwidth]{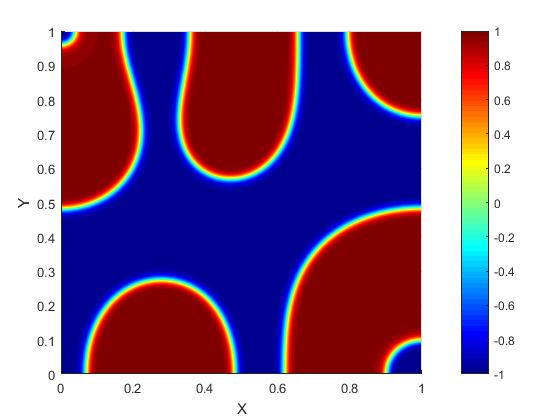}
    \end{minipage}
}
\subfigure[T=1.0]{
    \begin{minipage}[c]{0.26\textwidth}
    \includegraphics[width=1\textwidth]{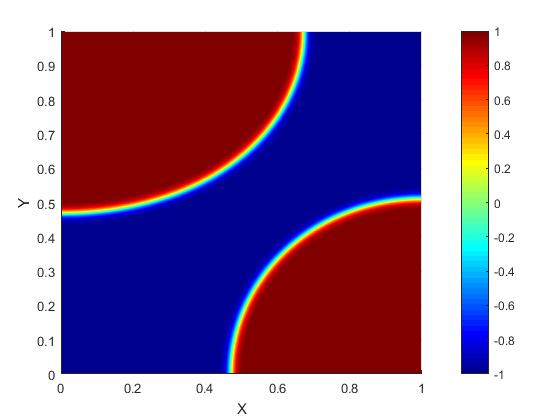}
    \end{minipage}
}
\rule{13cm}{0.05em}
\begin{minipage}[c]{0.14\linewidth}
Adaptive
\end{minipage}
\subfigure[T=0.02000]{
    \begin{minipage}[c]{0.26\textwidth}
    \includegraphics[width=1\textwidth]{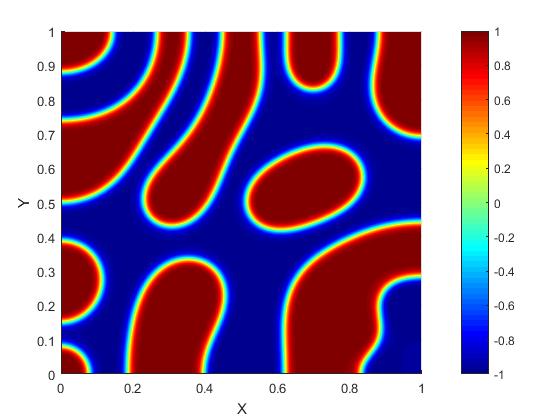}
    \end{minipage}
}
\subfigure[T=0.10000]{
    \begin{minipage}[c]{0.26\textwidth}
    \includegraphics[width=1\textwidth]{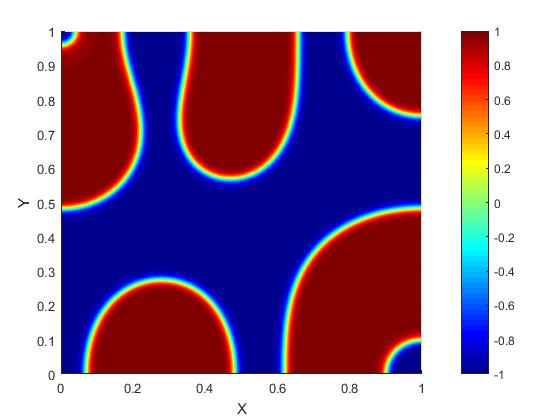}
    \end{minipage}
}
\subfigure[T=0.99990]{
    \begin{minipage}[c]{0.26\textwidth}
    \includegraphics[width=1\textwidth]{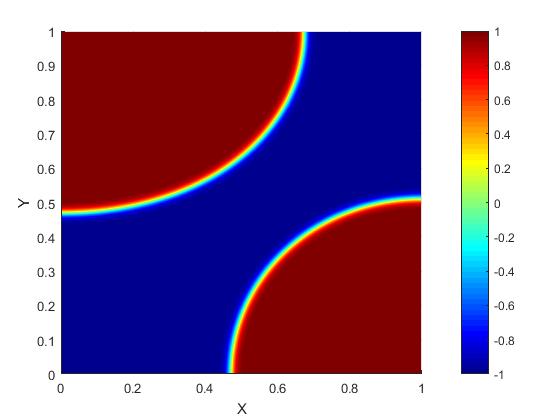}
    \end{minipage}
}
\rule{13cm}{0.05em}
\begin{minipage}[c]{0.14\linewidth}
$\Delta t=10^{-3}$
\end{minipage}
\subfigure[T=0.02]{
    \begin{minipage}[c]{0.26\textwidth}
    \includegraphics[width=1\textwidth]{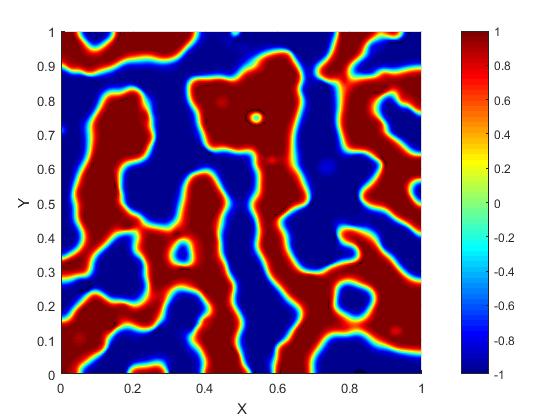}
    \end{minipage}
}
\subfigure[T=0.1]{
    \begin{minipage}[c]{0.26\textwidth}
    \includegraphics[width=1\textwidth]{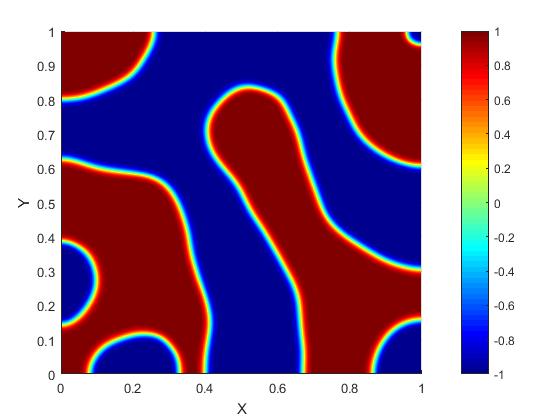}
    \end{minipage}
}
\subfigure[T=1.0]{
    \begin{minipage}[c]{0.26\textwidth}
    \includegraphics[width=1\textwidth]{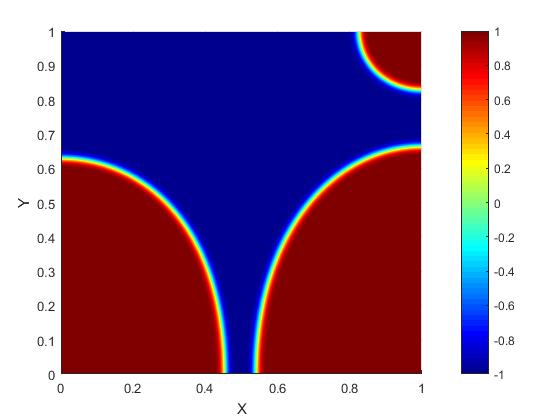}
    \end{minipage}
}
\caption{  {Snapshots of the phase function among small time steps, adaptive time steps and large time steps in example 3.} }\label{fig2_example5}
\end{figure}

\begin{figure}[!htp]
\centering
\includegraphics[scale=0.2]{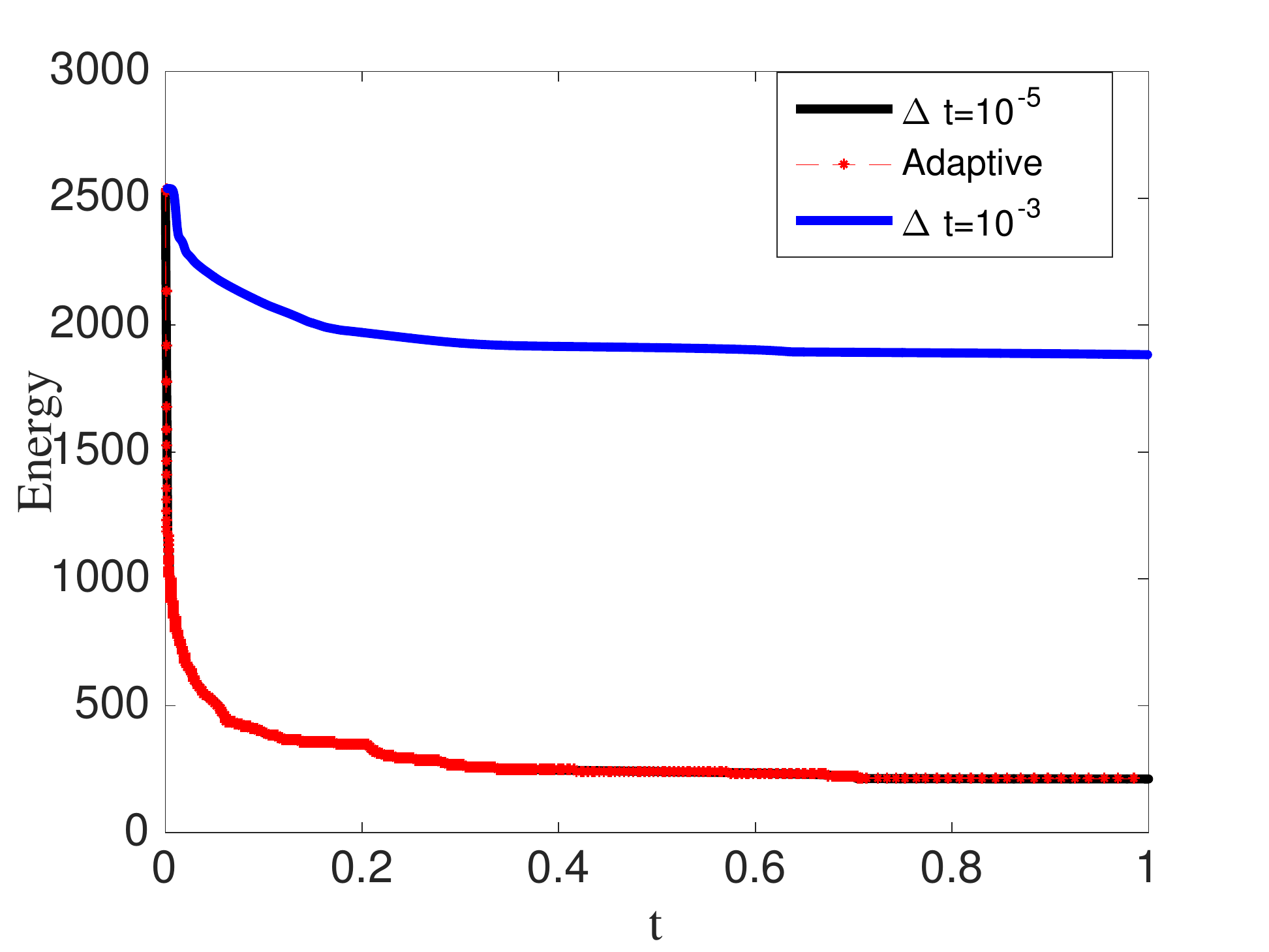}
\includegraphics[scale=0.2]{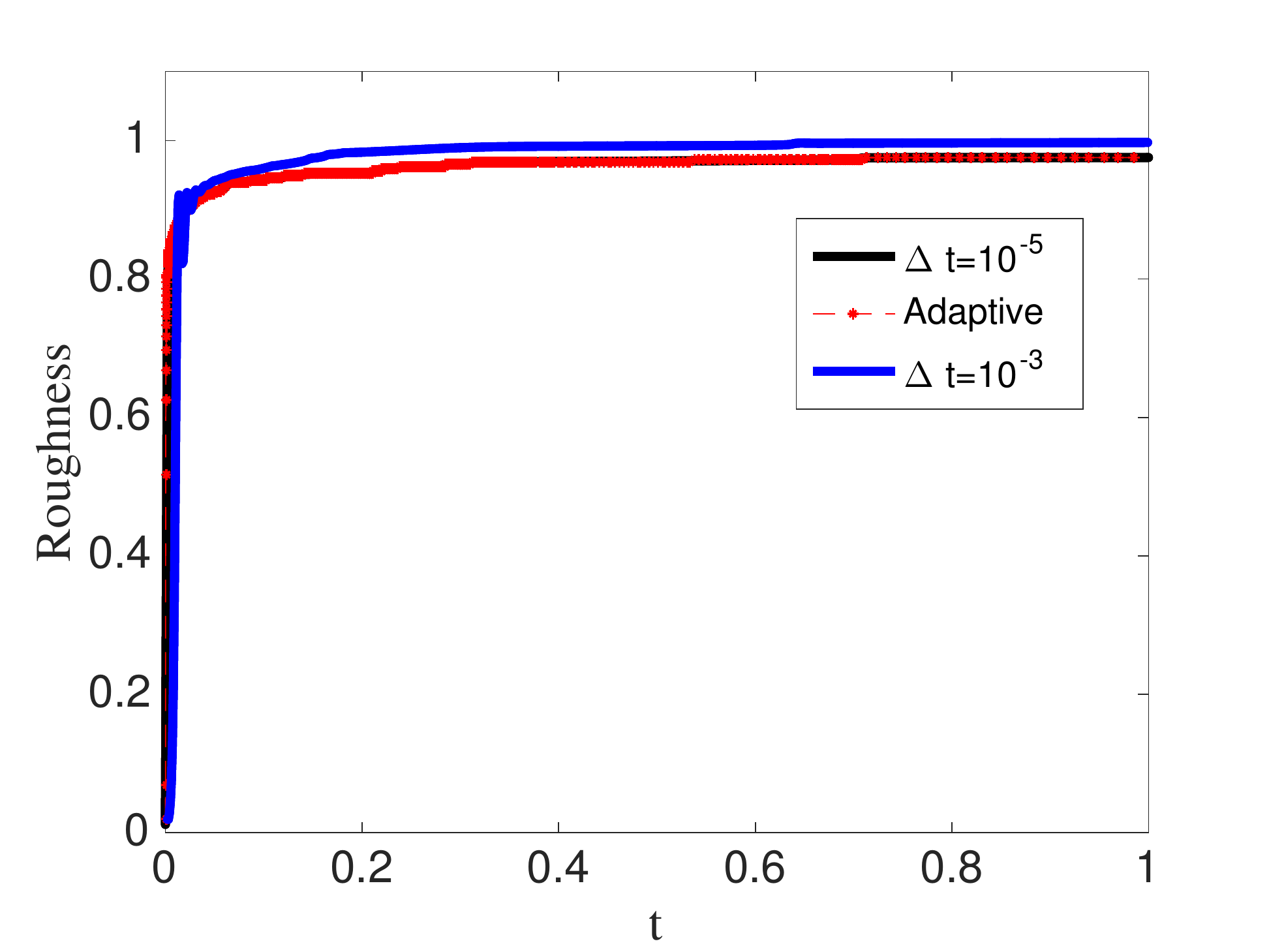}
\caption{Numerical comparisons of discrete scaled surface energy and roughness  for the simulation of spinodal decomposition in example 3.}  \label{fig3_example5}
\end{figure}

\begin{figure}[!htp]
\centering
\includegraphics[scale=0.2]{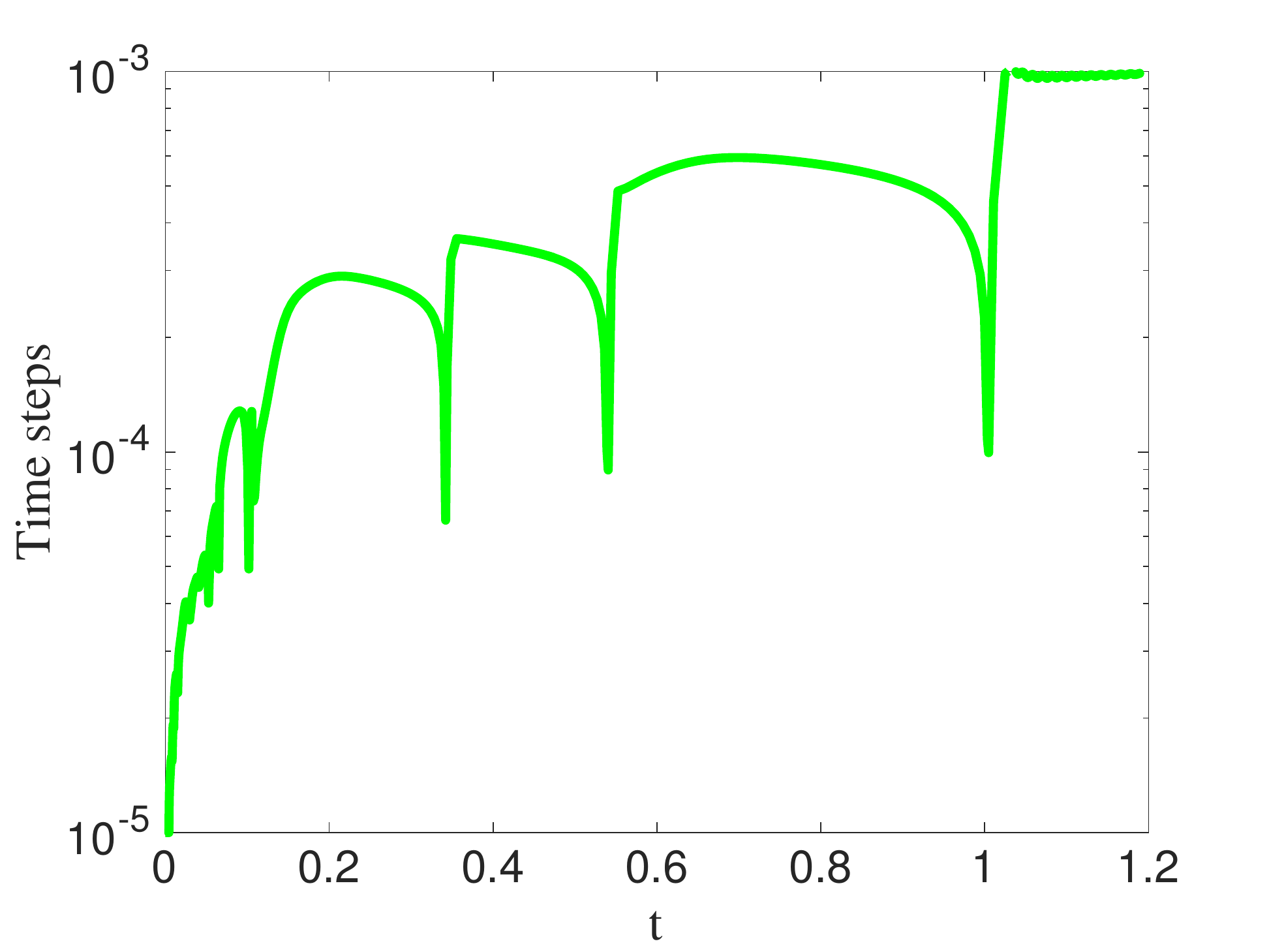}
\includegraphics[scale=0.2]{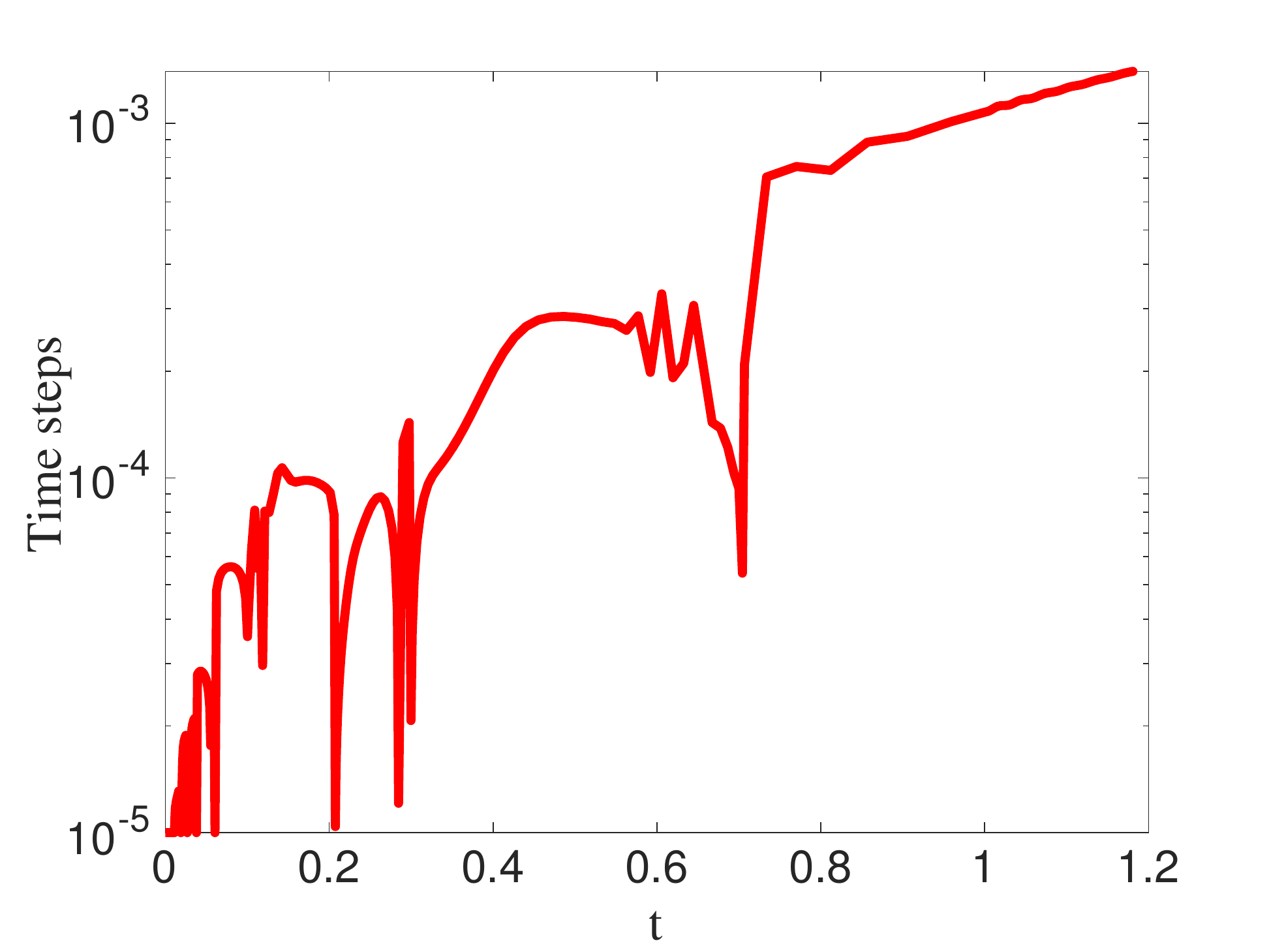}
\includegraphics[scale=0.2]{epsilon_02-eps-converted-to.pdf}
\caption{Adaptive time steps for different $\epsilon$: (a) $\epsilon=0.02$, (b) $\epsilon=0.01$, (c) $\epsilon=0.005$} \label{fig4_revised_different_epsilon}
\end{figure}


\section*{Acknowledgments}
 X. Li thanks for the financial support from China Scholarship Council. The work of J. Shen is supported in part by NSF grants  DMS-1620262, 
DMS-1720442 and AFOSR  grant FA9550-16-1-0102. The work of H. Rui is supported by the National Natural Science Foundation of China grant 11671233. 

\bibliographystyle{siamplain}
\bibliography{ex_article}
\end{document}